\documentclass[final,3p]{elsarticle}
 \usepackage{graphics}
 \usepackage{graphicx}
 \usepackage{epsfig}
\usepackage{amssymb}
 \usepackage{amsthm}
 \usepackage{lineno}
 \usepackage{amsmath}
   \numberwithin{equation}{section}
\usepackage{mathrsfs}
\usepackage{titlesec}

\titlespacing{\subsubsection}{0pt}{3ex plus 0.1ex minus .2ex}{1ex plus .2ex}
\journal{``Journal of Pseudo-Differential Operators and Applications"} 

\newtheorem{thm}{Theorem}[section]

\newtheorem{lem}[thm]{Lemma}

\newtheorem{defn}[thm]{Definition}
\newtheorem{rem}[thm]{Remark}
\newtheorem{exam}[thm]{Example}

\biboptions{sort&compress,square}
\allowdisplaybreaks
\begin{document}
\begin{frontmatter}
\author{Tong Wu$^{a}$}
\ead{wut977@nenu.edu.cn}
\author{Yong Wang$^{b,*}$}
\ead{wangy581@nenu.edu.cn}
\cortext[cor]{Corresponding author.}
\address{$^a$Department of Mathematics, Northeastern University, Shenyang, 110819, China}
\address{$^b$School of Mathematics and Statistics, Northeast Normal University,
Changchun, 130024, China}

\title{The pseudo-differential perturbations of the Dirac operator and the Kastler-Kalau-Walze type theorems}
\begin{abstract}
We define two types of pseudo-differential perturbations of the Dirac operator within the framework of the noncommutative geometry. And we obtain the noncommutative residue of the inverse square of these perturbations on 4-dimensional compact manifolds without boundary. As a generalization of the result of noncommutative residue on closed manifolds, we prove the Kastler-Kalau-Walze type theorems for these
perturbations on 4-dimensional compact manifolds with boundary. Finally, several examples in which we can consider the corresponding pseudo-differential operators and the Kastler-Kalau-Walze type theorems are listed.
\end{abstract}
\end{frontmatter}
\section{Introduction}
 Until now, many geometers have studied noncommutative residues. In \cite{Gu,Wo}, authors found noncommutative residues are of great importance to the study of the noncommutative geometry. In \cite{Co1}, Connes used the noncommutative residue to derive a conformal 4-dimensional Polyakov action analogy. Connes showed us that the noncommutative residue on a compact manifold $M$ coincided with the Dixmier's trace on pseudodifferential operators of order $-{\rm {dim}}M$ in \cite{Co2}.
And Connes claimed the noncommutative residue of the square of the inverse of the Dirac operator was proportioned to the Einstein-Hilbert action.  Kastler \cite{Ka} gave a
brute-force proof of this theorem. Kalau and Walze proved this theorem in the normal coordinates system simultaneously in \cite{KW}.
Ackermann proved that
the Wodzicki residue  of the square of the inverse of the Dirac operator ${\rm  Wres}(D^{-2})$ in turn is essentially the second coefficient
of the heat kernel expansion of $D^{2}$ in \cite{Ac}.

On the other hand, Wang generalized the Connes' results to the case of manifolds with boundary in \cite{Wa1,Wa2},
and proved the Kastler-Kalau-Walze type theorems for the Dirac operator and the signature operator on lower-dimensional manifolds
with boundary \cite{Wa3}. In \cite{Wa3,Wa4}, denote the noncommutative residue of manifolds with boundary by $\widetilde{{\rm Wres}}$, Wang computed $\widetilde{{\rm Wres}}[\pi^+D^{-1}\circ\pi^+D^{-1}]$ and $\widetilde{{\rm Wres}}[\pi^+D^{-2}\circ\pi^+D^{-2}]$, where two operators are symmetric, in these cases the boundary term vanished. But for $\widetilde{{\rm Wres}}[\pi^+D^{-1}\circ\pi^+D^{-3}]$, Wang got a nonvanishing boundary term \cite{Wa5}, and give a theoretical explanation for the gravitational action on boundary. In others words, Wang provides a kind of method to study the Kastler-Kalau-Walze type theorem for manifolds with boundary.

Many perturbations of the Dirac operator by zero-order differential operators are studied in \cite{L1,L2,W1,W3,W4,W2}. Wang proved a Kastler-Kalau-Walze type theorem for the perturbations of the Dirac operator on compact manifolds with (without) boundary and gave two kinds of operator-theoretic explanations of the gravitational action on boundary in \cite{W1}. In \cite{W4}, Wang, Wang and Yang gave two kinds of operator-theoretic
explanations of the gravitational action on 4-dimensional compact manifolds with
flat boundary. Furthermore, they got the Kastler-Kalau-Walze type theorems for 4-dimensional complex manifolds associated with nonminimal operators. In \cite{W2}, Wu, Wang and Wang gave the proofs of the Kastler-Kalau-Walze type theorems for Dirac-Witten operators on 4-dimensional and 6-dimensional compact manifolds with boundary. In \cite{L2}, Liu, Wu and Wang proved the Kastler-Kalau-Walze type theorems for the perturbation of de Rham Hodge operators on 4-dimensional and 6-dimensional compact
manifolds with boundary. In \cite{W3}, Wang, Wang and Wu gave some new
spectral functionals which is the extension of spectral functionals to the noncommutative realm with torsion, and related them to the
noncommutative residue for manifolds with boundary. In \cite{L1},  Li, Wu and Wang established some general Kastler-Kalau-Walze type theorems for any dimensional manifolds with boundary. All of the above studies are the noncommutative residue for laplacian operators. $\mathbf{The~motivation}$ of this paper is
to study the noncommutative residue for non-laplacian operators within the framework of the noncommutative geometry. We consider two types of pseudo-differential perturbations of the Dirac operator whose square are not laplacians and establish the Kastler-Kalau-Walze type theorems for these perturbations on 4-dimensional compact oriented spin manifolds without boundary.\\
\indent Actually, for the Dirac operator, the noncommutative residue of the
square of the inverse of the Dirac operator was proportioned to the Einstein-Hilbert action. In order to prove that the existence of noncommutative residues of some pseudo-differential perturbations of the Dirac operator does not satisfy the above conclusion. We calculate ${\rm Wres}(D+c(X)D^{-1}fD)^{-2}$ and ${\rm Wres}(D+D^{-1}c(X)D)^{-2}$, where ${\rm Wres}$ denotes the noncommutative residue of manifolds without boundary, $f\in C^\infty(M)$ and $c(X)$ is a Clifford action of the vector field $X$ on $M$. This constitutes a proof of the Kastler-Kalau-Walze type theorems for the pseudo-differential perturbations of the Dirac operator $D+c(X)D^{-1}fD$ and $D+D^{-1}c(X)D$ on 4-dimensional manifolds without boundary. Further, we examine several examples, including the Hodge-Dirac triple, almost-commutative two-sheeted space, conformally rescaled noncommutative tori, showing that these noncommutative examples in which we can consider the corresponding pseudo-differential operators and the Kastler-Kalau-Walze type theorems.
Our main theorems are as follows.
\begin{thm}\label{thm1111}
Let $M$ be a $4$-dimensional oriented
compact spin manifold without boundary, then we get the noncommutative residue of the operator $D+c(X)D^{-1}fD$
\begin{align*}
&{\rm Wres}(D+c(X)D^{-1}fD)^{-2}=4\int_{M}\bigg(2\pi^2div_M(fX)+2\pi^2X(df)-2\pi^2|X||df|+\frac{1}{12}s\bigg)d{\rm Vol_{M}},
\end{align*}
where $div_M$ denotes divergence of $M$, $s$ is the scalar curvature and ${\rm Vol_{M}}$ is the volume of $M$.
\end{thm}
\begin{thm}\label{thm2222}
Let $M$ be a $4$-dimensional oriented
compact spin manifold without boundary, then we get the noncommutative residue of the operator $D+D^{-1}c(X)D$
\begin{align*}
&{\rm Wres}(D+D^{-1}c(X)D)^{-2}=4\int_{M}\bigg(-4\pi^2|X|^2+3\pi^2div_M(X)+\frac{1}{12}s\bigg) d{\rm Vol_{M}}.
\end{align*}
\end{thm}
\begin{rem}
We can generalize Theorem \ref{thm1111} and Theorem \ref{thm2222} to the examples in Section \ref{section:4}.
\end{rem}
\indent The paper is organized in the following way. In Section \ref{section:2}, firstly, by using the formula of the noncommutative residue of the pseudo-differential operator $P$, we obtain the noncommutative residue of the inverse square of the first type pseudo-differential perturbation of the Dirac operator on 4-dimensional mani-
folds without boundary. Next, we recall some basic facts and formulas about Boutet de
Monvel's calculus and prove the Kastler-Kalau-Walze type theorem of the first type pseudo-differential perturbation on 4-dimensional manifolds with boundary. In Section \ref{section:3},
 we prove the Kastler-Kalau-Walze type theorem of the second type pseudo-differential perturbation of the Dirac operator on 4-dimensional manifolds with boundary. In Section \ref{section:4}, we introduce several examples of spectral triples, all of which are important examples of the noncommutative geometry.
\section{The first type pseudo-differential perturbation of the Dirac operator $D+c(X)D^{-1}fD$}
\label{section:2}
In this section, on the basis of the Dirac operator, we define the first type pseudo-differential perturbation of the Dirac operator $D+c(X)D^{-1}fD$, and prove the Kastler-Kalau-Walze type theorem of the operator $D+c(X)D^{-1}fD$ on 4-dimensional manifolds with boundary.
\subsection{${\rm Wres}(D+c(X)D^{-1}fD)^{-2}$ on manifolds without boundary}
Let $M$ be a 4-dimensional compact oriented spin manifold with Riemannian metric $g$, and let $\nabla^L$ be the Levi-Civita connection about $g$.
We recall that the Dirac operator $D$ is locally given as follows in terms of an orthonormal section $e_i$ (with dual section $\theta^k$) of the frame bundle of M \cite{Ka}:
\begin{align*}
&D=i\gamma^i\widetilde{\nabla}_i=i\gamma^i(e_i+\sigma_i);\nonumber\\
&\sigma_i(x)=\frac{1}{4}\gamma_{ij,k}(x)\gamma^i\gamma^k=\frac{1}{8}\gamma_{ij,k}(x)(\gamma^j\gamma^k-\gamma^k\gamma^j),
\end{align*}
where $\gamma_{ij,k}$ represents the Levi-Civita connection $\nabla$ with spin connection $\widetilde{\nabla}$, specifically:
\begin{align*}
&\gamma_{ij,k}=-\gamma_{ik,j}=\frac{1}{2}(c_{ij,k}+c_{ki,j}+c_{kj,i}),~~~i,j,k=1,\cdot\cdot\cdot,4;\nonumber\\
&c_{ij}^k=\theta^k([e_i.e_j]).
\end{align*}
Here the $\gamma^i$ are constant self-adjoint Dirac matrices s.t. $\gamma^i\gamma^j+\gamma^j\gamma^i=-2\delta^{ij}.$ In terms
of local coordinates $x^\mu$ inducing the alternative vierbein $\partial_\mu=S_\mu^i(x)e_i$ (with dual
vierbein $dx^\mu$), we have $\gamma^ie_i=\gamma^\mu \partial_\mu$, the $\gamma^\mu$  being now $x$-dependent Dirac matrices s.t. $\gamma^\mu\gamma^\nu+\gamma^\nu\gamma^\mu=-2g^{\mu\nu}$ (we use latin sub-(super-) scripts for the basic $e_i$ and greek
sub-(super-) scripts for the basis $\partial_\mu$, the type of sub-(super-) scripts specifying the type
of Dirac matrices). The specification of the Dirac operator in the greek basis is as
follows: one has
\begin{align*}
&D=i\gamma^\mu\widetilde{\nabla}_\mu=i\gamma^\mu(e_\mu+\sigma_\mu);\nonumber\\
&\sigma_\mu(x)=S_\mu^i(x)\sigma_i.
\end{align*}
Now, we define the first type pseudo-differential perturbation of the Dirac operator, set $A=D+c(X)D^{-1}fD$, where $c(X)$ is a Clifford action on $M$ and $X=\sum_{\alpha=1}^na_{\alpha}e_\alpha=\sum_{j=1}^nX_j\partial_j$ is a vector field, $f$ is a $C^\infty(M)$ function.
By further expansion, we have
\begin{align*}
A&=D+c(X)D^{-1}(Df+[f,D])=D+c(X)f-c(X)D^{-1}c(df);\\
A^2&=D^2+Dc(X)f+c(X)fD+[c(X)f]^2-Dc(X)D^{-1}c(df)-c(X)fc(X)D^{-1}c(df)\nonumber\\
&-c(X)D^{-1}c(df)D-c(X)D^{-1}c(df)c(X)f+[c(X)D^{-1}c(df)]^2.
\end{align*}
Since we know that $A^2$ is not a generalized laplacian, so we can not use the heat kernel coefficient to compute its residue. Next, our computations are based on the algorithm yielding the principal symbol of a
product of pseudo-differential operators in terms of the principal symbols of the
factors, namely, with the shorthand $\partial^\alpha_\xi=\partial^\alpha/\partial\xi_\alpha,~~\partial_\alpha^x=\partial_\alpha/\partial x^\alpha:$
\begin{align}\label{1111}
\sigma^{PQ}(x,\xi)=\sum_\alpha\frac{(-i)^\alpha}{\alpha!}\partial^\alpha_\xi\sigma^{P}(x,\xi)\cdot\partial_\alpha^x\sigma^{Q}(x,\xi).
\end{align}
We need to compute the total symbol $\sigma(x,\xi)$ of $A^{-2}$ up to order -4, with $A^2$ the
following sum of terms $A^2_k$ of order $k$:
$$
A^2=(A^2)_2+(A^2)_1+(A^2)_0+(A^2)_{<0}.
$$
\begin{align}\label{2222}
\sigma^{A^2}_2(x,\xi)&=|\xi|^2;\nonumber\\
\sigma^{A^2}_1(x,\xi)&=ifc(\xi)c(X)+ifc(X)c(\xi)+i(\Gamma^\mu-2\sigma^\mu)\xi_\mu;\nonumber\\
\sigma^{A^2}_0(x,\xi)&=-(\partial^x\sigma_\mu+\sigma^\mu\sigma_\mu-\Gamma^\mu\sigma_\mu)+\frac{1}{4}s+if\gamma^\mu\sigma_\mu c(X)+ifc(X)\gamma^\mu\sigma_\mu\nonumber\\
&-|\xi|^{-2}c(\xi)c(X)c(\xi)c(df)-|\xi|^{-2}c(X)c(\xi)c(df)c(\xi)+f^2|X|^2.
\end{align}
We want to compute a parametrix $A^{-2}$ of $A^2$ up to order $-4$ using the above
recipe: this amounts to computing the parts $\sigma^{A^{-2}}_{-k},k=2,3,4,$ in the expansion of the
full symbol $\sigma$ of $A^2$ into terms of decreasing order:
\begin{align*}
\sigma^{A^{-2}}=\sigma^{A^{-2}}_{-2}+\sigma^{A^{-2}}_{-3}+\sigma^{A^{-2}}_{-4}+terms~~of~~order \leq -5.
\end{align*}
Application of (\ref{1111}) with $P=A^2$ and $Q=A^{-2}$ yields in the respective orders
$0,- 1,-2$ the recurrence relations:
\begin{align}\label{1221}
\sigma^{A^{-2}}_{-2}&=(\sigma^{A^{2}}_2)^{-1};\nonumber\\
\sigma^{A^{-2}}_{-3}&=-\sigma^{A^{-2}}_{-2}[\sigma^{A^{2}}_1\sigma^{A^{-2}}_{-2}-i\partial_{\xi}^\mu\sigma^{A^{2}}_2\partial_\mu^x\sigma^{A^{-2}}_{-2}];\nonumber\\
\sigma^{A^{-2}}_{-4}&=-\sigma^{A^{-2}}_{-2}[\sigma^{A^{2}}_1\sigma^{A^{-2}}_{-3}+\sigma^{A^{2}}_0\sigma^{A^{-2}}_{-2}-i\partial_{\xi}^\mu\sigma^{A^{2}}_1\partial_\mu^x\sigma^{A^{-2}}_{-2}-i\partial_{\xi}^\mu\sigma^{A^{2}}_2\partial_\mu^x\sigma^{A^{-2}}_{-3}].
\end{align}
Put (\ref{2222}) into (\ref{1221}), we have
\begin{align}\label{3355}
\sigma^{A^{-2}}_{-2}&=|\xi|^{-2};\nonumber\\
\sigma^{A^{-2}}_{-3}&=-|\xi|^{-2}[(ifc(\xi)c(X)+ifc(X)c(\xi)+i(\Gamma^\mu-2\sigma^\mu)\xi_\mu)|\xi|^{-2}-i\partial^\mu_\xi(|\xi|^2)\partial_\mu^x(|\xi|^{-2})];\nonumber\\
\sigma^{A^{-2}}_{-4}&=-|\xi|^{-6}\xi_\mu\xi_\nu(\Gamma^\mu-2\sigma^\mu)(\Gamma^\nu-2\sigma^\nu)-2|\xi|^{-8}\xi^\mu\xi_\alpha\xi_\beta\xi_\nu(\Gamma^\nu-2\sigma^\nu)\partial^x_\mu g^{\alpha\beta}+|\xi|^{-4}(\partial^{x\mu}\sigma_\mu+\sigma^\mu\sigma_\mu\nonumber\\
&-\Gamma^\mu\sigma_\mu)-\frac{1}{4}|\xi|^{-4}s-2i|\xi|^{-2}\xi^\mu\cdot\partial^x_\mu\sigma_{-3}+|\xi|^{-6}\xi_\alpha\xi_\beta(\Gamma^\mu-2\sigma^\mu)\partial^x_\mu g^{\alpha\beta}-|\xi|^{-6}\xi_\alpha\xi_\beta g^{\mu\nu}\partial^x_{\mu\nu} g^{\alpha\beta}\nonumber\\
&+2|\xi|^{-8}\xi_\alpha\xi_\beta\xi_\gamma\xi_\delta g^{\mu\nu}\partial^x_\mu g^{\alpha\beta}\partial^x_\nu g^{\gamma\delta}-|\xi|^{-6}f(c(\xi)c(X)+c(X)c(\xi))\xi_\mu(\Gamma^\mu-2\sigma^\mu)-|\xi|^{-4}c(\xi)c(X)\nonumber\\
&c(\xi)c(df)+2|\xi|^{-8}f(c(\xi)c(X)+c(X)c(\xi))\xi^\mu\xi_\alpha\xi_\beta\partial^x_\mu g^{\alpha\beta}-|\xi|^{-6}f^2(c(\xi)c(X)+c(X)c(\xi))^2+|\xi|^{-2}\nonumber\\
&f^2|X|^2-|\xi|^{-6}f\xi_\mu(\Gamma^\mu-2\sigma^\mu)(c(\xi)c(X)+c(X)c(\xi))+i|\xi|^{-2}f(\gamma^\mu\sigma_\mu c(X)+c(X)\gamma^\mu\sigma_\mu )-|\xi|^{-4}\nonumber\\
&c(X)c(\xi)c(df)c(\xi)+|\xi|^{-4}\partial^\mu_\xi[f(c(\xi)c(X)+c(X)c(\xi))]\xi_\alpha\xi_\beta\partial^x_\mu g^{\alpha\beta},
\end{align}
where $s$ is the scalar curvature.

Regrouping the terms and inserting:
\begin{align*}
\partial^x_\mu\sigma^{A^{-2}}_{-3}&=2i|\xi|^{-6}\xi_\nu\xi_\alpha\xi_\beta(\Gamma^\nu-2\sigma^\nu)\partial^x_\mu g^{\alpha\beta}-i|\xi|^{-4}\xi_\nu\partial^x_\mu(\Gamma^\nu-2\sigma^\nu)+6i|\xi|^{-8}\xi_\nu\xi_\alpha\xi_\beta\xi_\gamma\xi_\delta\partial^x_\mu g^{\alpha\beta}\partial^x_\nu g^{\gamma\delta}\nonumber\\
&-2i|\xi|^{-6}\xi_\alpha\xi_\gamma\xi_\delta\partial^x_\mu g^{\nu\alpha}\partial^x_\nu g^{\gamma\delta}-2i|\xi|^{-6}\xi^\nu\xi_\gamma\xi_\delta\partial^x_{\mu\nu}g^{\gamma\delta}-i\partial^x_\mu [|\xi|^{-4}f(c(\xi)c(X)+c(X)c(\xi))].
\end{align*}
We get for $\sigma^{A^{-2}}_{-4}$ the sum of terms:
\begin{align*}
N_1&=-|\xi|^{-6}\xi_\mu\xi_\nu\Gamma^\mu\Gamma^\nu+|\xi|^{-4}[g_{\mu\nu}-|\xi|^{-4}\xi_{\mu\nu}][\sigma^\mu\sigma^\nu-\Gamma^\nu\sigma^\nu];~~~N_2=|\xi|^{-4}\partial^{x\mu}\sigma_\mu-\frac{1}{4}|\xi|^{-4}s;\nonumber\\
N_3&=-6|\xi|^{-8}\xi^\mu\xi_\nu\xi_\alpha\xi_\beta(\Gamma^\nu-2\sigma^\nu)\partial^x_\mu  g^{\alpha\beta};~~~
N_4=2|\xi|^{-6}\xi^\mu\xi_\nu\partial^x_\mu(\Gamma^\nu-2\sigma^\nu);\nonumber\\
N_5&=-12|\xi|^{-10}\xi^\mu\xi^\nu\xi_\alpha\xi_\beta\xi_\gamma\xi_\delta\partial^x_\mu  g^{\alpha\beta}\partial^x_\nu g^{\gamma\delta};~~~
N_6=4|\xi|^{-8}\xi^\mu\xi_\alpha\xi_\gamma\xi_\delta\partial^x_\mu g^{\nu\alpha}\partial^x_\nu g^{\gamma\delta};\nonumber\\
N_7&=|\xi|^{-6}\xi_\alpha\xi_\beta(\Gamma^\mu-2\sigma^\mu)\partial^x_\mu  g^{\alpha\beta};~~~
N_8=4|\xi|^{-8}\xi^\mu\xi^\nu\xi_\gamma\xi_\delta\partial^x_{\mu\nu}  g^{\gamma\delta};\nonumber\\
N_9&=-|\xi|^{-6}\xi_\alpha\xi_\beta g^{\mu\nu}\partial^x_{\mu\nu}  g^{\alpha\beta};~~~
N_{10}=2|\xi|^{-8}\xi_\alpha\xi_\beta\xi_\gamma\xi_\delta g^{\mu\nu}\partial^x_\mu  g^{\alpha\beta}\partial^x_\nu g^{\gamma\delta},
\end{align*}
and
\begin{align*}
M_1&=-|\xi|^{-6}f[c(\xi)c(X)+c(X)c(\xi)]\xi_\mu(\Gamma^\mu-2\sigma^\nu);~~
M_2=2|\xi|^{-8}f[c(\xi)c(X)+c(X)c(\xi)]\xi^\mu\xi_\alpha\xi_\beta\partial^x_\mu  g^{\alpha\beta};\nonumber\\
M_3&=-|\xi|^{-6}f^2[c(\xi)c(X)+c(X)c(\xi)]^2;~~~
M_4=-|\xi|^{-6}f\xi_\mu(\Gamma^\mu-2\sigma^\nu)[c(\xi)c(X)+c(X)c(\xi)];\nonumber\\
M_5&=-2|\xi|^{-2}\xi^\mu\partial^x_\nu [|\xi|^{-2}f(c(\xi)c(X)+c(X)c(\xi))];~~~
M_6=i|\xi|^{-2}f[\gamma^\mu\sigma_\mu c(X)+c(X)\gamma^\mu\sigma_\mu ];\nonumber\\
M_7&=-|\xi|^{-4}\partial_\xi^\mu[f(c(\xi)c(X)+c(X)c(\xi))]\xi_\alpha\xi_\beta\partial^x_\mu  g^{\alpha\beta};~~~
M_8=-|\xi|^{-4}c(\xi)c(X)c(\xi)c(df);\nonumber\\
M_9&=-|\xi|^{-4}c(X)c(\xi)c(df)c(\xi);~~~
M_{10}=|\xi|^{-2}f^2|X|^2.
\end{align*}

For a pseudo-differential operator $P$, acting on sections of a vector bundle over an $n$-dimensional compact
Riemannian manifold $M$ in \cite{Ka}, it is defined by
\begin{align}\label{666}
{\rm Wres}(P):=\int_{S^*M}{\rm tr}(\sigma_{-n}^P)(x,\xi),
\end{align}
where $\xi\in S^{n-1},$ ${\rm tr}$ is shorthand of trace.

By (\ref{666}), we get the noncommutative residue of the operator $D+c(X)D^{-1}fD$ on 4-dimensional manifolds without boundary. Next, we need to compute the following formula
\begin{align*}
{\rm Wres}(D+c(X)D^{-1}fD)^{-2}=\int_{M}\int_{|\xi|=1}{\rm tr}(\sigma^{A^{-2}}_{-4})(x_0)\sigma(\xi)dx,
\end{align*}
where $x_0\in M$ is any fixed point.

Because from \cite{Ka}, we obtain $\int_{|\xi|=1}{\rm tr}[\sum_{i=1}^{10}N_i](x_0)\sigma(\xi)=-\frac{1}{12}s{\rm tr}[{\rm \texttt{id}}].$
Then, we only need to compute $$\int_{|\xi|=1}{\rm tr}[\sum_{i=1}^{10}M_i](x_0)\sigma(\xi).$$

$\mathbf{(1):}$
 By $c(\xi)=c(\xi^*),$ $\xi(X)=g(\xi^*,X)$ and $c(\xi)c(X)+c(X)c(\xi)=-2\xi(X)$, we have
\begin{align}
\label{a26}
&{\rm tr}(M_1)_{|\xi|=1}=2f\xi(X)\xi_\mu{\rm tr}(g^{\alpha\beta}\Gamma^\mu_{\alpha\beta}-2\sigma^\mu).
\end{align}
Using the facts:
$\Gamma^\mu_{\alpha\beta}(x_0)=\sigma_\mu(x_0)=0,$ then in normal coordinates, the result of this term $M_1$ disappears.

$\mathbf{(2):}$
\begin{align}\label{yyy}
&{\rm tr}(M_2)_{|\xi|=1}=-4f\xi(X)\xi^\mu\xi_\alpha\xi_\beta{\rm tr}(\partial_\mu^xg^{\alpha\beta}).
\end{align}
Using the facts:
$\partial_\mu^xg^{\alpha\beta}(x_0)=0,$ then in normal coordinates, the result of this term $M_2$ disappears.

$\mathbf{(3):}$ By $e_j^*(x_0)=dx_j$ and $\int_{|\xi|=1}\xi_j\xi_l\sigma(\xi)=\frac{1}{4}\delta_{jl}Vol_{S^3}=\frac{\pi^2}{2}\delta_{jl}$, we have
\begin{align}\label{pppppp}
{\rm tr}(M_3)_{|\xi|=1}&=-4f^2\xi(X)^2\nonumber\\
&=-4f^2\sum_j\xi_jdx_j\sum_l\xi_ldx_l{\rm tr}[{\rm \texttt{id}}]\nonumber\\
&=-4f^2\sum_{j,l}\xi_j\xi_le_j^*(x_0)e_l^*(x_0){\rm tr}[{\rm \texttt{id}}],
\end{align}
and
\begin{align}\label{lll}
\int_{|\xi|=1}{\rm tr}(M_3)(x_0)\sigma(\xi)&=-4f^2\frac{\pi^2}{2}\sum_je_j^*(x_0)e_j^*(x_0){\rm tr}[{\rm \texttt{id}}]\nonumber\\
&=-2f^2\pi^2|X|^2{\rm tr}[{\rm \texttt{id}}].
\end{align}
$\mathbf{(4):}$ By $\mathbf{(1)}$ and $\mathbf{(2)}$, in normal coordinates, the result of this term $M_4,M_6,M_7$ disappear.\\
$\mathbf{(5):}$
\begin{align*}
{\rm tr}(M_5)(x_0)=4\sum_{\mu,l}\xi_\mu\xi_l{\rm tr}[\partial_\mu^x(dx_l(x)f)](x_0).
\end{align*}
Then
\begin{align*}
\int_{|\xi|=1}{\rm tr}(M_5)(x_0)\sigma(\xi)&=2\pi^2{\rm tr}[\sum_l\partial^*_l(dx_l(fX))](x_0)\nonumber\\
&=2\pi^2{\rm tr}[\sum_le_l(g\langle e_l,fX\rangle)](x_0)\nonumber\\
&=2\pi^2div_M(fX){\rm tr}[{\rm \texttt{id}}],
\end{align*}
where $div_M$ denotes divergence of $M$.\\
$\mathbf{(6):}$ By ${\rm tr}(ab)={\rm tr}(ba),$ we have
\begin{align*}
{\rm tr}(M_8)(x_0)={\rm tr}(M_9)(x_0)&={\rm tr}[-c(\xi)c(X)c(\xi)c(df)]\nonumber\\
&={\rm tr}[2g(X,\xi)c(\xi)c(df)-c(X)c(df)]\nonumber\\
&={\rm tr}[2\xi(X)c(\xi)c(df)-c(X)c(df)]\nonumber\\
&={\rm tr}[-2\xi(X)\xi(df)+X(df)],
\end{align*}
then
\begin{align*}
\int_{|\xi|=1}{\rm tr}[X(df)]\sigma(\xi)=2\pi^2X(df){\rm tr}[{\rm \texttt{id}}],
\end{align*}
and
\begin{align*}
\int_{|\xi|=1}{\rm tr}[-2\xi(X)\xi(df)]\sigma(\xi)&=\int_{|\xi|=1}-2\sum_j\xi_je_j^*(x)\sum_l\xi_le_l^*(df)\sigma(\xi)\nonumber\\
&=-2\pi^2|X||df|{\rm tr}[{\rm \texttt{id}}].
\end{align*}
$\mathbf{(7):}$
\begin{align}\label{a2666}
\int_{|\xi|=1}{\rm tr}(M_{10})\sigma(\xi)&=\int_{|\xi|=1}f^2|X|^2\sigma(\xi)\nonumber\\
&=2f^2\pi^2|X|^2{\rm tr}[{\rm \texttt{id}}].
\end{align}
Thus
\begin{align*}
\int_{|\xi|=1}{\rm tr}(\sigma^{A^{-2}}_{-4})(x_0)\sigma(\xi)=\bigg(2\pi^2div_M(fX)+2\pi^2X(df)-2\pi^2|X||df|+\frac{1}{12}s\bigg){\rm tr}[{\rm \texttt{id}}].
\end{align*}
When $n=4$, ${\rm tr}_{S(TM)}[{\rm \texttt{id}}]=4$, we have the following result.
\begin{thm}\label{thm2}Let $M$ be a $4$-dimensional oriented
compact spin manifold without boundary, then we get the noncommutative residue of the operator $D+c(X)D^{-1}fD$
\begin{align*}
&{\rm Wres}(D+c(X)D^{-1}fD)^{-2}=4\int_{M}\bigg(2\pi^2div_M(fX)+2\pi^2X(df)-2\pi^2|X||df|+\frac{1}{12}s\bigg)d{\rm Vol_{M}},
\end{align*}
where ${\rm Vol_{M}}$ is the volume of $M$.
\end{thm}
\subsection{${\rm \widetilde{Wres}}[\pi^+(D+c(X)D^{-1}fD)^{-1}\circ \pi^+(D+c(X)D^{-1}fD)^{-1}]$ on manifolds with boundary}
\subsubsection{Boutet de
Monvel's calculus}
 In this subsubsection, we recall some basic facts and formulas about Boutet de
Monvel's calculus and the definition of the noncommutative residue for manifolds with boundary which will be used in the followings. For more details, see Section 2 in \cite{Wa3}.\\
 \indent Let $M$ be an n-dimensional compact oriented spin manifold with boundary $\partial M$.
We assume that the metric $g$ on $M$ has the following form near the boundary,
$$
g=\frac{1}{h(x_{n})}g^{\partial M}+dx _{n}^{2},
$$
where $g^{\partial M}$ is the metric on $\partial M$, and $h(x_n)\in C^{\infty}([0, 1)):=\{\widehat{h}|_{[0,1)}|\widehat{h}\in C^{\infty}((-\varepsilon,1))\}$ for
some $\varepsilon>0$ and $h(x_n)$ satisfies $h(x_n)>0$, $h(0)=1$, where $x_n$ denotes the normal directional coordinate. Let $U\subset M$ be a collar neighborhood of $\partial M$ which is diffeomorphic with $\partial M\times [0,1)$. By the definition of $h(x_n)\in C^{\infty}([0,1))$
and $h(x_n)>0$, there exists $\widehat{h}\in C^{\infty}((-\varepsilon,1))$ such that $\widehat{h}|_{[0,1)}=h$ and $\widehat{h}>0$ for some
sufficiently small $\varepsilon>0$. Thus, on $\widetilde{M}=M\bigcup_{\partial M}\partial M\times
(-\varepsilon,0]$, there exists a metric $g'$ which has the following form on $U\bigcup_{\partial M}\partial M\times (-\varepsilon,0 ]$,
\begin{equation*}
g'=\frac{1}{\widehat{h}(x_{n})}g^{\partial M}+dx _{n}^{2} ,
\end{equation*}
such that $g'|_{M}=g$. And we fix a metric $g'$ on the $\widetilde{M}$ such that $g'|_{M}=g$.

Let the Fourier transformation $F'$  be
\begin{equation*}
F':L^2({\bf R}_t)\rightarrow L^2({\bf R}_v);~F'(u)(v)=\int_\mathbb{R} e^{-ivt}u(t)dt
\end{equation*}
and let
\begin{equation*}
r^{+}:C^\infty ({\bf R})\rightarrow C^\infty (\widetilde{{\bf R}^+});~ f\rightarrow f|\widetilde{{\bf R}^+};~
\widetilde{{\bf R}^+}=\{x\geq0;x\in {\bf R}\}.
\end{equation*}
Here $u(t)$ is a complex value function, not a real value function.

We define $H^+=F'(\Phi(\widetilde{{\bf R}^+}));~ H^-_0=F'(\Phi(\widetilde{{\bf R}^-}))$ which satisfies
$H^+\bot H^-_0$, where $\Phi(\widetilde{{\bf R}^+}) =r^+\Phi({\bf R})$, $\Phi(\widetilde{{\bf R}^-}) =r^-\Phi({\bf R})$ and $\Phi({\bf R})$
denotes the Schwartz space. We have the following
 property: $h\in H^+~$ (resp. $H^-_0$) if and only if $h\in C^\infty({\bf R})$ which has an analytic extension to the lower (resp. upper) complex
half-plane $\{{\rm Im}\xi<0\}$ (resp. $\{{\rm Im}\xi>0\})$ such that for all nonnegative integer $l$,
 \begin{equation}\label{pqow}
\frac{d^{l}h}{d\xi^l}(\xi)\sim\sum^{\infty}_{k=1}\frac{d^l}{d\xi^l}(\frac{c_k}{\xi^k}),
\end{equation}
as $|\xi|\rightarrow +\infty,{\rm Im}\xi\leq0$ (resp. ${\rm Im}\xi\geq0)$ and where $c_k\in\mathbb{C}$ are some constants.\\
When $|\xi|\rightarrow +\infty$ Eq. (\ref{pqow}) means
\begin{align*}
\frac{d^{l}h}{d\xi^l}(\xi)-\sum^{n}_{k=1}\frac{d^l}{d\xi^l}(\frac{c_k}{\xi^k})=o(|\xi|^{-(n+l)}).
\end{align*}
When $|\xi|\rightarrow +\infty$ for a constant $C_{n,l}$, Eq. (\ref{pqow}) also means
\begin{align*}
\frac{d^{l}h}{d\xi^l}(\xi)-\sum^{n}_{k=1}\frac{d^l}{d\xi^l}(\frac{c_k}{\xi^k})\leq C_{n,l}(|\xi|)^{-(n+l+1)}.
\end{align*}

Let $H'$ be the space of all polynomials and $H^-=H^-_0\bigoplus H';~H=H^+\bigoplus H^-.$ Denote by $\pi^+$ (resp. $\pi^-$) the
 projection on $H^+$ (resp. $H^-$). Let $\widetilde H=\{$rational functions having no poles on the real axis$\}$. Then on $\tilde{H}$,
 \begin{equation}\label{b3}
\pi^+h(\xi_0)=\frac{1}{2\pi i}\lim_{u\rightarrow 0^{-}}\int_{\Gamma^+}\frac{h(\xi)}{\xi_0+iu-\xi}d\xi,
\end{equation}
where $\Gamma^+$ is a Jordan closed curve
included ${\rm Im}(\xi)>0$ surrounding all the singularities of $h$ in the upper half-plane and
$\xi_0\in {\bf R}$. In our computations, we only compute $\pi^+h$ for $h$ in $\widetilde{H}$. Similarly, define $\pi'$ on $\tilde{H}$,
\begin{equation}\label{b4}
\pi'h=\frac{1}{2\pi}\int_{\Gamma^+}h(\xi)d\xi.
\end{equation}
So $\pi'(H^-)=0$. For $h\in H\bigcap L^1({\bf R})$, $\pi'h=\frac{1}{2\pi}\int_{{\bf R}}h(v)dv$ and for $h\in H^+\bigcap L^1({\bf R})$, $\pi'h=0$.\\
\indent An operator of order $m\in {\bf Z}$ and type $d$ is a matrix\\
$$\widetilde{A}=\left(\begin{array}{lcr}
  \pi^+P+G  & K  \\
   T  &  \widetilde{S}
\end{array}\right):
\begin{array}{cc}
\   C^{\infty}(M,E_1)\\
 \   \bigoplus\\
 \   C^{\infty}(\partial{M},F_1)
\end{array}
\longrightarrow
\begin{array}{cc}
\   C^{\infty}(M,E_2)\\
\   \bigoplus\\
 \   C^{\infty}(\partial{M},F_2)
\end{array},
$$
where $M$ is a manifold with boundary $\partial M$ and
$E_1,E_2$~ (resp. $F_1,F_2$) are vector bundles over $M~$ (resp. $\partial M
$).~Here,~$P:C^{\infty}_0(\Omega,\overline {E_1})\rightarrow
C^{\infty}(\Omega,\overline {E_2})$ is a classical
pseudodifferential operator of order $m$ on $\Omega$, where
$\Omega$ is a collar neighborhood of $M$ and
$\overline{E_i}|M=E_i~(i=1,2)$. $P$ has an extension:
$~{\cal{E'}}(\Omega,\overline {E_1})\rightarrow
{\cal{D'}}(\Omega,\overline {E_2})$, where
${\cal{E'}}(\Omega,\overline {E_1})~({\cal{D'}}(\Omega,\overline
{E_2}))$ is the dual space of $C^{\infty}(\Omega,\overline
{E_1})~(C^{\infty}_0(\Omega,\overline {E_2}))$. Let
$e^+:C^{\infty}(M,{E_1})\rightarrow{\cal{E'}}(\Omega,\overline
{E_1})$ denote extension by zero from $M$ to $\Omega$ and
$r^+:{\cal{D'}}(\Omega,\overline{E_2})\rightarrow
{\cal{D'}}(\Omega, {E_2})$ denote the restriction from $\Omega$ to
$X$, then define
$$\pi^+P=r^+Pe^+:C^{\infty}(M,{E_1})\rightarrow {\cal{D'}}(\Omega,
{E_2}).$$ In addition, $P$ is supposed to have the
transmission property; this means that, for all $j,k,\alpha$, the
homogeneous component $p_j$ of order $j$ in the asymptotic
expansion of the
symbol $p$ of $P$ in local coordinates near the boundary satisfies:\\
$$\partial^k_{x_n}\partial^\alpha_{\xi'}p_j(x',0,0,+1)=
(-1)^{j-|\alpha|}\partial^k_{x_n}\partial^\alpha_{\xi'}p_j(x',0,0,-1),$$
then $\pi^+P:C^{\infty}(M,{E_1})\rightarrow C^{\infty}(M,{E_2})$
by Theorem 4 in \cite{RS} page 139. Let $G$,$T$ be respectively the singular Green operator
and the trace operator of order $m$ and type $d$. Let $K$ be a
potential operator and $S$ be a classical pseudodifferential
operator of order $m$ along the boundary. Denote by $B^{m,d}$ the collection of all operators of
order $m$
and type $d$,  and $\mathcal{B}$ is the union over all $m$ and $d$.\\
\indent Recall that $B^{m,d}$ is a Fr\'{e}chet space. The composition
of the above operator matrices yields a continuous map:
$B^{m,d}\times B^{m',d'}\rightarrow B^{m+m',{\rm max}\{
m'+d,d'\}}.$ Write $$\widetilde{A}=\left(\begin{array}{lcr}
 \pi^+P+G  & K \\
 T  &  \widetilde{S}
\end{array}\right)
\in B^{m,d},
 \widetilde{A}'=\left(\begin{array}{lcr}
\pi^+P'+G'  & K'  \\
 T'  &  \widetilde{S}'
\end{array} \right)
\in B^{m',d'}.$$\\
 The composition $\widetilde{A}\widetilde{A}'$ is obtained by
multiplication of the matrices (For more details see \cite{SE}). For
example, $\pi^+P\circ G'$ and $G\circ G'$ are singular Green
operators of type $d'$ and
$$\pi^+P\circ\pi^+P'=\pi^+(PP')+L(P,P').$$
Here $PP'$ is the usual
composition of pseudodifferential operators and $L(P,P')$ called
leftover term is a singular Green operator of type $m'+d$. In this case, $P,P'$ are classical pseudo differential operators, in other words $\pi^+P\in \mathcal{B}^{\infty}$ and $\pi^+P'\in \mathcal{B}^{\infty}$ .

Denote by $\mathcal{B}$ the Boutet de Monvel's algebra. We recall that the main theorem in \cite{FGLS,Wa3}.
\begin{thm}\label{th:32}{\rm\cite{FGLS}}{\bf(Fedosov-Golse-Leichtnam-Schrohe)}
 Let $M$ and $\partial M$ be connected, ${\rm dim}M=n\geq3$, and let $\widetilde{S}$ (resp. $\widetilde{S}'$) be the unit sphere about $\xi$ (resp. $\xi'$) and $\sigma(\xi)$ (resp. $\sigma(\xi')$) be the corresponding canonical
$n-1$ (resp. $(n-2)$) volume form.
 Set $\widetilde{A}=\left(\begin{array}{lcr}\pi^+P+G &   K \\
T &  \widetilde{S}    \end{array}\right)$ $\in \mathcal{B}$ , and denote by $p$, $b$ and $s$ the local symbols of $P,G$ and $\widetilde{S}$ respectively.
 Define:
 \begin{align*}
{\rm{\widetilde{Wres}}}(\widetilde{A})&=\int_X\int_{\bf \widetilde{ S}}{\rm{tr}}_E\left[p_{-n}(x,\xi)\right]\sigma(\xi)dx \nonumber\\
&+2\pi\int_ {\partial X}\int_{\bf \widetilde{S}'}\left\{{\rm tr}_E\left[({\rm{tr}}b_{-n})(x',\xi')\right]+{\rm{tr}}
_F\left[s_{1-n}(x',\xi')\right]\right\}\sigma(\xi')dx',
\end{align*}
where ${\rm{\widetilde{Wres}}}$ denotes the noncommutative residue of an operator in the Boutet de Monvel's algebra.\\
Then~~ a) ${\rm \widetilde{Wres}}([\widetilde{A},B])=0 $, for any
$\widetilde{A},B\in\mathcal{B}$;~~ b) It is the unique continuous trace on
$\mathcal{B}/\mathcal{B}^{-\infty}$.
\end{thm}
\subsubsection{The boundary term of ${\rm \widetilde{Wres}}[\pi^+(D+c(X)D^{-1}fD)^{-1}\circ \pi^+(D+c(X)D^{-1}fD)^{-1}]$}
In this subsubsection, we want to calculate the boundary term of ${\rm \widetilde{Wres}}[\pi^+(D+c(X)D^{-1}fD)^{-1}\circ \pi^+(D+c(X)D^{-1}fD)^{-1}]$ on manifolds with boundary. Firstly, according to \cite{Wa3}, we get the following definition, which shows the relationship between lower dimensional volumes of spin manifolds and the noncommutative residue on manifolds with boundary.
\begin{defn}\label{def1}{\rm\cite{Wa3} }
Lower dimensional volumes of spin manifolds with boundary are defined by
 \begin{align*}
{\rm Vol}^{(p_1,p_2)}_nM:= \widetilde{{\rm Wres}}[\pi^+D^{-p_1}\circ\pi^+D^{-p_2}],
\end{align*}
\end{defn}
and
\begin{align}
\label{b7}
\widetilde{{\rm Wres}}[\pi^+D^{-p_1}\circ\pi^+D^{-p_2}]=\int_M\int_{|\xi|=1}{\rm
tr}_{\wedge^*T^*M\bigotimes\mathbb{C}}[\sigma_{-n}(D^{-p_1-p_2})]\sigma(\xi)dx+\int_{\partial M}\Phi,
\end{align}
where
\begin{align}
\label{b8}
\Phi&=\int_{|\xi'|=1}\int^{+\infty}_{-\infty}\sum^{\infty}_{j, k=0}\sum\frac{(-i)^{|\alpha|+j+k+1}}{\alpha!(j+k+1)!}
\times {\rm tr}_{\wedge^*T^*M\bigotimes\mathbb{C}}[\partial^j_{x_n}\partial^\alpha_{\xi'}\partial^k_{\xi_n}\sigma^+_{r}(D^{-p_1})(x',0,\xi',\xi_n)
\nonumber\\
&\times\partial^\alpha_{x'}\partial^{j+1}_{\xi_n}\partial^k_{x_n}\sigma_{l}(D^{-p_2})(x',0,\xi',\xi_n)]d\xi_n\sigma(\xi')dx',
\end{align}
and the sum is taken over $r+l-k-|\alpha|-j-1=-n,~~r\leq -p_1,~~l\leq -p_2$.

For any fixed point $x_0\in\partial M$, we choose the normal coordinates
$U$ of $x_0$ in $\partial M$ (not in $M$) and compute $\Phi(x_0)$ in the coordinates $\widetilde{U}=U\times [0,1)\subset M$ and with the
metric $\frac{1}{h(x_n)}g^{\partial M}+dx_n^2.$ The dual metric of $g$ on $\widetilde{U}$ is ${h(x_n)}g^{\partial M}+dx_n^2.$  Write
$g_{ij}=g(\partial_{x_i},\partial_{x_j})$ and $g^{ij}=g(dx_i,dx_j)$, then
\begin{equation*}
[g_{ij}]= \left[\begin{array}{lcr}
  \frac{1}{h(x_n)}[g_{ij}^{\partial M}]  & 0  \\
   0  &  1
\end{array}\right];~~~
[g^{ij}]= \left[\begin{array}{lcr}
  h(x_n)[g^{ij}_{\partial M}]  & 0  \\
   0  &  1
\end{array}\right],
\end{equation*}
and
\begin{equation*}
\partial_{x_s}g_{ij}^{\partial M}(x_0)=0,~~1\leq i,~~j\leq n-1;~~g_{ij}(x_0)=\delta_{ij}.
\end{equation*}
\indent From \cite{Wa3}, we can get three lemmas.
\begin{lem}{\rm \cite{Wa3}}\label{le:32}
With the metric $g$ on $M$ near the boundary, then
\begin{align*}
\partial_{x_j}(|\xi|_{g}^2)(x_0)&=\left\{
       \begin{array}{c}
        0,  ~~~~~~~~~~ ~~~~~~~~~~ ~~~~~~~~~~~~~{\rm if }~j<n, \\[2pt]
       h'(0)|\xi'|^{2}_{g^{\partial M}},~~~~~~~~~~~~~~~~~~~~{\rm if }~j=n,
       \end{array}
    \right. \\
\partial_{x_j}[c(\xi)](x_0)&=\left\{
       \begin{array}{c}
      0,  ~~~~~~~~~~ ~~~~~~~~~~ ~~~~~~~~~~~~~{\rm if }~j<n,\\[2pt]
\partial_{x_n}(c(\xi'))(x_{0}), ~~~~~~~~~~~~~~~~~{\rm if }~j=n,
       \end{array}
    \right.
\end{align*}
where $\xi=\xi'+\xi_{n}dx_{n}$.
\end{lem}
\begin{lem}{\rm \cite{Wa3}}\label{le:32}With the metric $g$ on $M$ near the boundary, then
\begin{align*}
\omega_{s,t}(e_i)(x_0)&=\left\{
       \begin{array}{c}
        \omega_{n,i}(e_i)(x_0)=\frac{1}{2}h'(0),  ~~~~~~~~~~ ~~~~~~~~~~~{\rm if }~s=n,t=i,i<n, \\[2pt]
       \omega_{i,n}(e_i)(x_0)=-\frac{1}{2}h'(0),~~~~~~~~~~~~~~~~~~~{\rm if }~s=i,t=n,i<n,\\[2pt]
    \omega_{s,t}(e_i)(x_0)=0,~~~~~~~~~~~~~~~~~~~~~~~~~~~other~cases,~~~~~~~~~
       \end{array}
    \right.
\end{align*}
where $(\omega_{s,t})$ denotes the connection matrix of Levi-Civita connection $\nabla^L$.
\end{lem}
\begin{lem}{\rm \cite{Wa3}}\label{lem1} When $i<n,$ then
\begin{align*}
\Gamma_{st}^k(x_0)&=\left\{
       \begin{array}{c}
        \Gamma^n_{ii}(x_0)=\frac{1}{2}h'(0),~~~~~~~~~~ ~~~~~~~~~~~{\rm if }~s=t=i,k=n, \\[2pt]
        \Gamma^i_{ni}(x_0)=-\frac{1}{2}h'(0),~~~~~~~~~~~~~~~~~~~{\rm if }~s=n,t=i,k=i,\\[2pt]
        \Gamma^i_{in}(x_0)=-\frac{1}{2}h'(0),~~~~~~~~~~~~~~~~~~~{\rm if }~s=i,t=n,k=i,\\[2pt]
       \end{array}
    \right.
\end{align*}
in other cases, $\Gamma_{st}^i(x_0)=0$.
\end{lem}

 On manifolds with boundary, $X$ can be expressed as $X=X^T+X_n\partial_{x_n}$. Thus, some symbols for the first type pseudo-differential perturbation of the Dirac operator are given by (\ref{2222}) and (\ref{3355}). Then we get the following lemmas.
\begin{lem}\label{lem2} Some symbols of positive order for the first type pseudo-differential perturbation of the Dirac operator are as follows.
\begin{align*}
\sigma_1(D+c(X)D^{-1}fD)&=\sigma_1(D)=ic(\xi); \nonumber\\
\sigma_0(D+c(X)D^{-1}fD)&=\sigma_0(D)+c(X)f=-\frac{1}{4}\sum_{i,s,t}\omega_{s,t}(e_i)c(e_i)c(e_s)c(e_t)+c(X)f.
\end{align*}
\end{lem}
Then by the composition formula of pseudo-differential operators, we have
\begin{lem}\label{lem3} Some symbols of negative order for the first type pseudo-differential perturbation of the Dirac operator are as follows.
\begin{align*}
\sigma_{-1}({D+c(X)D^{-1}fD})^{-1}&=\frac{ic(\xi)}{|\xi|^2};\nonumber\\
\sigma_{-2}({D+c(X)D^{-1}fD})^{-1}&=\frac{c(\xi)\sigma_{0}({D+c(X)D^{-1}fD})c(\xi)}{|\xi|^4}+\frac{c(\xi)}{|\xi|^6}\sum_jc(dx_j)
\Big[\partial_{x_j}(c(\xi))|\xi|^2-c(\xi)\partial_{x_j}(|\xi|^2)\Big].
\end{align*}
\end{lem}
Further, by (\ref{b7}) and (\ref{b8}), we obtain lower dimensional volumes of spin manifolds with boundary for the operator $D+c(X)D^{-1}fD$
\begin{align}
\label{b14}
&\widetilde{{\rm Wres}}[\pi^+(D+c(X)D^{-1}fD)^{-1}\circ \pi^+(D+c(X)D^{-1}fD)^{-1}]\nonumber\\
&=\int_M\int_{|\xi|=1}{\rm
tr}_{S(TM)\bigotimes\mathbb{C}}[\sigma_{-4}(D+c(X)D^{-1}fD)^{-2}]\sigma(\xi)dx+\int_{\partial M}\Phi,
\end{align}
where
\begin{align}
\label{b15}
\Phi &=\int_{|\xi'|=1}\int^{+\infty}_{-\infty}\sum^{\infty}_{j, k=0}\sum\frac{(-i)^{|\alpha|+j+k+1}}{\alpha!(j+k+1)!}
\times {\rm
tr}_{S(TM)\bigotimes\mathbb{C}}[\partial^j_{x_n}\partial^\alpha_{\xi'}\partial^k_{\xi_n}\sigma^+_{r}({D+c(X)D^{-1}fD})^{-1}(x',0,\xi',\xi_n)
\nonumber\\
&\times\partial^\alpha_{x'}\partial^{j+1}_{\xi_n}\partial^k_{x_n}\sigma_{l}({D+c(X)D^{-1}fD})^{-1}(x',0,\xi',\xi_n)]d\xi_n\sigma(\xi')dx',
\end{align}
and the sum is taken over $r+l-k-j-|\alpha|=-3,~~r\leq -1,~l\leq-1$.

Since $[\sigma_{-n}(D^{-p_1-p_2})]|_M$ has the same expression as $\sigma_{-n}(D^{-p_1-p_2})$ in the case of manifolds without
boundary, so locally we can compute the first term by \cite{Ka}, \cite{KW}, \cite{Po}, \cite{Wa3}. Therefore, the interior term of $\widetilde{{\rm Wres}}[\pi^+(D+c(X)D^{-1}fD)^{-1}\circ \pi^+(D+c(X)D^{-1}fD)^{-1}]$ has been given in Theorem \ref{thm2}. After computing the boundary term of $\widetilde{{\rm Wres}}[\pi^+(D+c(X)D^{-1}fD)^{-1}\circ \pi^+(D+c(X)D^{-1}fD)^{-1}]$, thai is $\int_{\partial M} \Phi$, we get the following theorem.
\begin{thm}\label{thmb1}
Let $M$ be a $4$-dimensional oriented
compact spin manifold with boundary, then we get the Kastler-Kalau-Walze type theorem of the operator $D+c(X)D^{-1}fD$
\begin{align*}
&\widetilde{{\rm Wres}}[\pi^+({D+c(X)D^{-1}fD})^{-1}\circ\pi^+({D+c(X)D^{-1}fD})^{-1}]\nonumber\\
&=4\int_{M}\bigg(2\pi^2div_M(fX)+2\pi^2X(df)-2\pi^2|X||df|+\frac{1}{12}s\bigg)d{\rm Vol_{M}}.
\end{align*}
In particular, the boundary term vanishes.
\end{thm}
\begin{proof}
\indent When $n=4$, the sum is taken over $
r+l-k-j-|\alpha|=-3,~~r\leq -1,~~l\leq-1,$ then we have the following five cases:
~\\
\noindent  {\bf case a-I)}~$r=-1,~l=-1,~k=j=0,~|\alpha|=1$.\\
\noindent By (\ref{b15}), we get
\begin{align*}
\Phi_1&=-\int_{|\xi'|=1}\int^{+\infty}_{-\infty}\sum_{|\alpha|=1}
 {\rm tr}[\partial^\alpha_{\xi'}\pi^+_{\xi_n}\sigma_{-1}({D+c(X)D^{-1}fD})^{-1}\nonumber\\
 &\times
 \partial^\alpha_{x'}\partial_{\xi_n}\sigma_{-1}({D+c(X)D^{-1}fD})^{-1}](x_0)d\xi_n\sigma(\xi')dx'.
\end{align*}
By Lemma \ref{le:32}, for $i<n$, then
\begin{align*}
\partial_{x_i}\left(\frac{ic(\xi)}{|\xi|^2}\right)(x_0)=
\frac{i\partial_{x_i}[c(\xi)](x_0)}{|\xi|^2}
-\frac{ic(\xi)\partial_{x_i}(|\xi|^2)(x_0)}{|\xi|^4}=0,
\end{align*}
\noindent so $\Phi_1=0$.\\
 \noindent  {\bf case a-II)}~$r=-1,~l=-1,~k=|\alpha|=0,~j=1$.\\
\noindent By (\ref{b15}), we get
\begin{align*}
\Phi_2&=-\frac{1}{2}\int_{|\xi'|=1}\int^{+\infty}_{-\infty} {\rm
tr} [\partial_{x_n}\pi^+_{\xi_n}\sigma_{-1}({D+c(X)D^{-1}fD})^{-1}\nonumber\\
&\times
\partial_{\xi_n}^2\sigma_{-1}({D+c(X)D^{-1}fD})^{-1}](x_0)d\xi_n\sigma(\xi')dx'.
\end{align*}
By Lemma \ref{lem3} and taking further derivatives, we have
\begin{align*}\partial^2_{\xi_n}\sigma_{-1}({D+c(X)D^{-1}fD})^{-1}(x_0)=i\left(-\frac{6\xi_nc(dx_n)+2c(\xi')}
{|\xi|^4}+\frac{8\xi_n^2c(\xi)}{|\xi|^6}\right);
\end{align*}
\begin{align*}
\partial_{x_n}\sigma_{-1}({D+c(X)D^{-1}fD})^{-1}(x_0)=\frac{i\partial_{x_n}c(\xi')(x_0)}{|\xi|^2}-\frac{ic(\xi)|\xi'|^2h'(0)}{|\xi|^4}.
\end{align*}
\noindent By the relation of the Clifford action and ${\rm tr}{ab}={\rm tr }{ba}$, we have the equalities:
\begin{align*}
&{\rm tr}[c(\xi')c(dx_n)]=0;~~{\rm tr}[c(dx_n)^2]=-4;~~{\rm tr}[c(\xi')^2](x_0)|_{|\xi'|=1}=-4;\nonumber\\
&{\rm tr}[\partial_{x_n}c(\xi')c(dx_n)]=0;~~{\rm tr}[\partial_{x_n}c(\xi')c(\xi')](x_0)|_{|\xi'|=1}=-2h'(0).
\end{align*}
Thus
\begin{align*}
\Phi_2&=-\int_{|\xi'|=1}\int^{+\infty}_{-\infty}\frac{ih'(0)(\xi_n-i)^2}
{(\xi_n-i)^4(\xi_n+i)^3}d\xi_n\sigma(\xi')dx'\nonumber\\
&=-ih'(0)\Omega_3\int_{\Gamma^+}\frac{1}{(\xi_n-i)^2(\xi_n+i)^3}d\xi_ndx'\nonumber\\
&=-ih'(0)\Omega_32\pi i\bigg[\frac{1}{(\xi_n+i)^3}\bigg]^{(1)}\bigg|_{\xi_n=i}dx'\nonumber\\
&=-\frac{3}{8}\pi h'(0)\Omega_3dx',
\end{align*}
where ${\rm \Omega_{3}}$ is the canonical volume of $S^{2}.$\\
\noindent  {\bf case a-III)}~$r=-1,~l=-1,~j=|\alpha|=0,~k=1$.\\
\noindent By (\ref{b15}), we get
\begin{align*}
\Phi_3&=-\frac{1}{2}\int_{|\xi'|=1}\int^{+\infty}_{-\infty}
{\rm tr} [\partial_{\xi_n}\pi^+_{\xi_n}\sigma_{-1}({D+c(X)D^{-1}fD})^{-1}\\
&\times
\partial_{\xi_n}\partial_{x_n}\sigma_{-1}({D+c(X)D^{-1}fD})^{-1}](x_0)d\xi_n\sigma(\xi')dx'.
\end{align*}
By Lemma \ref{lem3} and taking further derivatives,, we have
\begin{align*}
\partial_{\xi_n}\partial_{x_n}\sigma_{-1}({D+c(X)D^{-1}fD})^{-1}(x_0)|_{|\xi'|=1}&=-ih'(0)\left[\frac{c(dx_n)}{|\xi|^4}-4\xi_n\frac{c(\xi')
+\xi_nc(dx_n)}{|\xi|^6}\right]\\
&-\frac{2\xi_ni\partial_{x_n}c(\xi')(x_0)}{|\xi|^4};
\end{align*}
\begin{align*}
\partial_{\xi_n}\pi^+_{\xi_n}\sigma_{-1}({D+c(X)D^{-1}fD}^{-1})(x_0)|_{|\xi'|=1}=-\frac{c(\xi')+ic(dx_n)}{2(\xi_n-i)^2}.
\end{align*}
Similar to {\rm case~a-II)}, we have
\begin{align*}
{\rm tr}\left\{\frac{c(\xi')+ic(dx_n)}{2(\xi_n-i)^2}\times
ih'(0)\left[\frac{c(dx_n)}{|\xi|^4}-4\xi_n\frac{c(\xi')+\xi_nc(dx_n)}{|\xi|^6}\right]\right\}
=2h'(0)\frac{i-3\xi_n}{(\xi_n-i)^4(\xi_n+i)^3}
\end{align*}
and
\begin{align*}
{\rm tr}\left[\frac{c(\xi')+ic(dx_n)}{2(\xi_n-i)^2}\times
\frac{2\xi_ni\partial_{x_n}c(\xi')(x_0)}{|\xi|^4}\right]
=\frac{-2ih'(0)\xi_n}{(\xi_n-i)^4(\xi_n+i)^2}.
\end{align*}
Then we obtain the following result
\begin{align*}
\Phi_3&=-\int_{|\xi'|=1}\int^{+\infty}_{-\infty}\frac{h'(0)(i-3\xi_n)}
{(\xi_n-i)^4(\xi_n+i)^3}d\xi_n\sigma(\xi')dx'
-\int_{|\xi'|=1}\int^{+\infty}_{-\infty}\frac{h'(0)i\xi_n}
{(\xi_n-i)^4(\xi_n+i)^2}d\xi_n\sigma(\xi')dx'\nonumber\\
&=-h'(0)\Omega_3\frac{2\pi i}{3!}\left[\frac{(i-3\xi_n)}{(\xi_n+i)^3}\right]^{(3)}\bigg|_{\xi_n=i}dx'+h'(0)\Omega_3\frac{2\pi i}{3!}\left[\frac{i\xi_n}{(\xi_n+i)^2}\right]^{(3)}\bigg|_{\xi_n=i}dx'\nonumber\\
&=\frac{3}{8}\pi h'(0)\Omega_3dx'.
\end{align*}
\noindent  {\bf case a-IV)}~$r=-2,~l=-1,~k=j=|\alpha|=0$.\\
\noindent By (\ref{b15}), we get
\begin{align*}
\Phi_4&=-i\int_{|\xi'|=1}\int^{+\infty}_{-\infty}{\rm tr} [\pi^+_{\xi_n}\sigma_{-2}({D+c(X)D^{-1}fD})^{-1}\times
\partial_{\xi_n}\sigma_{-1}({D+c(X)D^{-1}fD})^{-1}](x_0)d\xi_n\sigma(\xi')dx'.
\end{align*}
 By Lemma \ref{lem3}, we have
\begin{align*}
&\sigma_{-2}({D+c(X)D^{-1}fD})^{-1}(x_0)\nonumber\\
&=\frac{c(\xi)\sigma_{0}(D+c(X)D^{-1}fD)(x_0)c(\xi)}{|\xi|^4}+\frac{c(\xi)}{|\xi|^6}c(dx_n)
[\partial_{x_n}[c(\xi')](x_0)|\xi|^2-c(\xi)h'(0)|\xi|^2_{\partial
M}].
\end{align*}
Denote $Q(x_0)=-\frac{1}{4}\sum_{s,t,i}\omega_{s,t}(e_i)(x_{0})c(e_i)c(e_s)c(e_t),$ then by (\ref{b3}) and (\ref{b4}), we get
\begin{align*}
\pi^+_{\xi_n}\sigma_{-2}(D+c(X)D^{-1}fD)^{-1}|_{|\xi'|=1}&=\pi^+_{\xi_n}\Big[\frac{c(\xi)Q(x_0)c(\xi)}{(1+\xi_n^2)^2}\Big]+\pi^+_{\xi_n}
\Big[\frac{c(\xi)c(X)fc(\xi)}{(1+\xi_n^2)^2}\Big]
\nonumber\\
&+\pi^+_{\xi_n}\Big[\frac{c(\xi)c(dx_n)\partial_{x_n}[c(\xi')](x_0)}{(1+\xi_n^2)^2}-h'(0)\frac{c(\xi)c(dx_n)c(\xi)}{(1+\xi_n^{2})^3}\Big]\nonumber\\
&:=C_1-C_2+C_3,
\end{align*}
where
\begin{align}\label{53}
C_1&=\frac{-1}{4(\xi_n-i)^2}[(2+i\xi_n)c(\xi')Q_0^{2}(x_0)c(\xi')+i\xi_nc(dx_n)Q_0^{2}(x_0)c(dx_n)\nonumber\\
&+(2+i\xi_n)c(\xi')c(dx_n)\partial_{x_n}c(\xi')+ic(dx_n)Q_0^{2}(x_0)c(\xi')
+ic(\xi')Q_0^{2}(x_0)c(dx_n)-i\partial_{x_n}c(\xi')],
\end{align}
\begin{align}\label{54}
C_2&=\frac{h'(0)}{2}\left[\frac{c(dx_n)}{4i(\xi_n-i)}+\frac{c(dx_n)-ic(\xi')}{8(\xi_n-i)^2}
+\frac{3\xi_n-7i}{8(\xi_n-i)^3}[ic(\xi')-c(dx_n)]\right],
\end{align}
and
\begin{align*}
C_3&=\frac{2+i\xi_n}{4(\xi_n-i)^2}fc(\xi')c(X)c(\xi')+\frac{i}{4(\xi_n-i)^2}fc(\xi')c(X)c(dx_n)+\frac{i}{4(\xi_n-i)^2}fc(dx_n)c(X)c(\xi')\nonumber\\
&+\frac{i\xi_n}{4(\xi_n-i)^2}fc(dx_n)c(X)c(dx_n).
\end{align*}
Taking derivatives, we have
\begin{align}\label{50}
\partial_{\xi_n}\sigma_{-1}(D+c(X)D^{-1}fD)^{-1}=i\left[\frac{c(dx_n)}{1+\xi_n^2}-\frac{2\xi_nc(\xi')+2\xi_n^2c(dx_n)}{(1+\xi_n^2)^2}\right].
\end{align}
Moreover, by the relation of the Clifford action and ${\rm tr}{ab}={\rm tr }{ba}$, we have the equalities:
\begin{align*}
&{\rm tr}[c(\xi')c(X)c(\xi')c(dx_n)]=-4X_n;~~
{\rm tr}[c(\xi')c(X)c(\xi')c(\xi')]=4g(X,\xi');\nonumber\\
&{\rm tr }[c(dx_n)c(X)c(dx_n)c(dx_n)]=4X_n;
~~{\rm tr}[c(dx_n)c(X)c(\xi')c(\xi')c(dx_n)]=4g(X,\xi').
\end{align*}
By (\ref{53}) and (\ref{50}), we have
\begin{align*}{\rm tr }[C_1\times\partial_{\xi_n}\sigma_{-1}(D+c(X)D^{-1}fD)^{-1}]|_{|\xi'|=1}=
\frac{3ih'(0)}{2(\xi_n-i)^2(1+\xi_n^2)^2}+h'(0)\frac{\xi_n^2-i\xi_n-2}{2(\xi_n-i)(1+\xi_n^2)^2}.
\end{align*}
By (\ref{54}) and (\ref{50}), we have
\begin{align*}{\rm tr }[C_2\times\partial_{\xi_n}\sigma_{-1}(D+c(X)D^{-1}fD)^{-1}]|_{|\xi'|=1}
&=2ih'(0)\frac{-i\xi_n^2-\xi_n+4i}{4(\xi_n-i)^3(\xi_n+i)^2};
\end{align*}
\begin{align*}{\rm tr }[C_3\times\partial_{\xi_n}\sigma_{-1}(D+c(X)D^{-1}fD)^{-1}]|_{|\xi'|=1}
&=\frac{2i}{(\xi_n-i)^3(\xi_n+i)}fX_n-4\frac{\xi_n+i\xi_n^2}{(\xi_n-i)^4(\xi_n+i)^2}fX_n\nonumber\\
&+\frac{2}{(\xi_n-i)^3(\xi_n+i)}fg(X,\xi')+\frac{4i\xi_n-\xi_n^2}{(\xi_n-i)^4(\xi_n+i)^2}fg(X,\xi').
\end{align*}
When $i<n,~\int_{|\xi'|=1}\xi_{i_{1}}\xi_{i_{2}}\cdots\xi_{i_{2d+1}}\sigma(\xi')=0$,
so $g(X,\xi')$ has no contribution for computing {\rm case~a-IV)}. Then, we have
\begin{align*}
&-i\int_{|\xi'|=1}\int^{+\infty}_{-\infty}{\rm tr} [(C_1-C_2)\times
\partial_{\xi_n}\sigma_{-1}(D+c(X)D^{-1}fD)^{-1}](x_0)d\xi_n\sigma(\xi')dx'\nonumber\\
&=\Omega_3\int_{\Gamma^+}\frac{3h'(0)(\xi_n-i)+ih'(0)}{2(\xi_n-i)^3(\xi_n+i)^2}d\xi_ndx'\nonumber\\
&=\frac{9}{8}\pi h'(0)\Omega_3dx',
\end{align*}
and
\begin{align*}
&-i\int_{|\xi'|=1}\int^{+\infty}_{-\infty}{\rm tr} [C_3\times
\partial_{\xi_n}\sigma_{-1}(D+c(X)D^{-1}fD)^{-1}](x_0)d\xi_n\sigma(\xi')dx'\nonumber\\
&=-i\int_{|\xi'|=1}\int^{+\infty}_{-\infty}\left(\frac{2i}{(\xi_n-i)^3(\xi_n+i)}fX_n-4\frac{\xi_n+i\xi_n^2}{(\xi_n-i)^4(\xi_n+i)^2}fX_n\right)d\xi_n\sigma(\xi')dx'\nonumber\\
&=f\Omega_3X_n\int_{\Gamma^+}\frac{2}{(\xi_n-i)^3(\xi_n+i)}d\xi_ndx'+4if\Omega_3X_n\int_{\Gamma^+}\frac{\xi_n+i\xi_n^2}{(\xi_n-i)^4(\xi_n+i)^2}d\xi_ndx'\nonumber\\
&=f\Omega_3X_n\frac{2\pi i}{2!}\bigg[\frac{2}{(\xi_n+i)}\bigg]^{(2)}\bigg|_{\xi_n=i}dx'+4if\Omega_3X_n\frac{2\pi i}{3!}\bigg[\frac{\xi_n+i\xi_n^2}{(\xi_n+i)^2}\bigg]^{(3)}\bigg|_{\xi_n=i}dx'\nonumber\\
&=-fX_n\pi\Omega_3dx'.
\end{align*}
Finally, we get
$$
\Phi_4=\bigg(\frac{9}{8} h'(0)-fX_n\bigg)\pi\Omega_3dx'.
$$
\noindent {\bf  case a-V)}~$r=-1,~l=-2,~k=j=|\alpha|=0$.\\
By (\ref{b15}), we get
\begin{align*}
\Phi_5=-i\int_{|\xi'|=1}\int^{+\infty}_{-\infty}{\rm tr} [\pi^+_{\xi_n}\sigma_{-1}(D+c(X)D^{-1}fD)^{-1}\times
\partial_{\xi_n}\sigma_{-2}(D+c(X)D^{-1}fD)^{-1}](x_0)d\xi_n\sigma(\xi')dx'.
\end{align*}
By (\ref{b3}), (\ref{b4}) and Lemma \ref{lem3}, we have
\begin{align}\label{62}
\pi^+_{\xi_n}\sigma_{-1}(D+c(X)D^{-1}fD)^{-1}|_{|\xi'|=1}=\frac{c(\xi')+ic(dx_n)}{2(\xi_n-i)}.
\end{align}
Since
\begin{align*}
&\sigma_{-2}(D+c(X)D^{-1}fD)^{-1}(x_0)\nonumber\\
&=\frac{c(\xi)\sigma_{0}(D+c(X)D^{-1}fD)(x_0)c(\xi)}{|\xi|^4}+\frac{c(\xi)}{|\xi|^6}c(dx_n)
\bigg[\partial_{x_n}[c(\xi')](x_0)|\xi|^2-c(\xi)h'(0)|\xi|^2_{\partial_
M}\bigg].
\end{align*}
Moreover
\begin{align*}
&\partial_{\xi_n}\sigma_{-2}(D+c(X)D^{-1}fD)^{-1}(x_0)|_{|\xi'|=1}\nonumber\\
&=
\partial_{\xi_n}\bigg\{\frac{c(\xi)[Q(x_0)
+c(X)f]c(\xi)}{|\xi|^4}+\frac{c(\xi)}{|\xi|^6}c(dx_n)[\partial_{x_n}[c(\xi')](x_0)|\xi|^2-c(\xi)h'(0)]\bigg\}\nonumber\\
&=\partial_{\xi_n}\bigg\{\frac{[c(\xi)Q(x_0)]c(\xi)}{|\xi|^4}+\frac{c(\xi)}{|\xi|^6}c(dx_n)[\partial_{x_n}[c(\xi')](x_0)|\xi|^2-c(\xi)h'(0)]\bigg\}+\partial_{\xi_n}\bigg(\frac{c(\xi)fc(X)c(\xi)}{|\xi|^4}\bigg).\nonumber\\
\end{align*}
By computations, we have
\begin{align}\label{67}
\partial_{\xi_n}\bigg(\frac{c(\xi)c(X)fc(\xi)}{|\xi|^4}\bigg)&=-\frac{4\xi_n}{(1+\xi_n^2)^3}fc(\xi')c(X)c(\xi')
+\bigg(\frac{1}{(1+\xi_n^2)^2}-\frac{4\xi_n^2}{(1+\xi_n^2)^3}\bigg)\bigg(fc(\xi')c(X)c(dx_n)\nonumber\\
&+fc(dx_n)c(X)c(\xi')\bigg)+\bigg(\frac{2\xi_n}{(1+\xi_n^2)^2}-\frac{4\xi_n^3}{(1+\xi_n^2)^3}\bigg)fc(dx_n)c(X)c(dx_n).
\end{align}
We denote $$q_{-2}^{1}=\frac{c(\xi)Q(x_0)c(\xi)}{|\xi|^4}+\frac{c(\xi)}{|\xi|^6}c(dx_n)[\partial_{x_n}[c(\xi')](x_0)|\xi|^2-c(\xi)h'(0)],$$ then
\begin{align}\label{68}
\partial_{\xi_n}(q_{-2}^{1})&=\frac{1}{(1+\xi_n^2)^3}\bigg[(2\xi_n-2\xi_n^3)c(dx_n)Q(x_0)c(dx_n)
+(1-3\xi_n^2)c(dx_n)Q(x_0)c(\xi')\nonumber\\
&+(1-3\xi_n^2)c(\xi')Q(x_0)c(dx_n)
-4\xi_nc(\xi')Q(x_0)c(\xi')
+(3\xi_n^2-1){\partial}_{x_n}c(\xi')\nonumber\\
&-4\xi_nc(\xi')c(dx_n){\partial}_{x_n}c(\xi')
+2h'(0)c(\xi')+2h'(0)\xi_nc(dx_n)\bigg]+6\xi_nh'(0)\frac{c(\xi)c(dx_n)c(\xi)}{(1+\xi^2_n)^4}.
\end{align}
By (\ref{62}) and (\ref{67}), we have
\begin{align*}
&{\rm tr}[\pi^+_{\xi_n}\sigma_{-1}(D+c(X)D^{-1}fD)^{-1}\times
\partial_{\xi_n}\bigg(\frac{c(\xi)c(X)fc(\xi)}
{|\xi|^4}\bigg)](x_0)\nonumber\\
&=4\frac{1-3\xi_n^2+3i\xi_n-i\xi_n^3}{(\xi_n-i)^4(\xi_n+i)^3}fX_n+4\frac{i(1-3\xi_n^2)-3\xi_n+\xi_n^3}{(\xi_n-i)^4(\xi_n+i)^3}fg(X,\xi').
\end{align*}
By (\ref{62}) and (\ref{68}), we have
\begin{align*}
{\rm tr}[\pi^+_{\xi_n}\sigma_{-1}(D+c(X)D^{-1}fD)^{-1}\times
\partial_{\xi_n}(q^1_{-2})](x_0)
=\frac{3h'(0)(i\xi^2_n+\xi_n-2i)}{(\xi-i)^3(\xi+i)^3}
+\frac{12h'(0)i\xi_n}{(\xi-i)^3(\xi+i)^4}.
\end{align*}
Then
\begin{align*}
-i\Omega_3\int_{\Gamma_+}\bigg[\frac{3h'(0)(i\xi_n^2+\xi_n-2i)}
{(\xi_n-i)^3(\xi_n+i)^3}+\frac{12h'(0)i\xi_n}{(\xi_n-i)^3(\xi_n+i)^4}\bigg]d\xi_ndx'=
-\frac{9}{8}\pi h'(0)\Omega_3dx'.
\end{align*}
When $i<n,~\int_{|\xi'|=1}\xi_{i_{1}}\xi_{i_{2}}\cdots\xi_{i_{2d+1}}\sigma(\xi')=0$ and $g(X,\xi')$ has no contribution for computing {\rm case~a-V)}, we have
\begin{align*}
&-i\int_{|\xi'|=1}\int^{+\infty}_{-\infty}{\rm tr}[\pi^+_{\xi_n}\sigma_{-1}(D+c(X)D^{-1}fD)^{-1}\times
\partial_{\xi_n}\frac{c(\xi)c(X)fc(\xi)}
{|\xi|^4}](x_0)d\xi_n\sigma(\xi')dx'\nonumber\\
&=-i\int_{|\xi'|=1}\int^{+\infty}_{-\infty}4\frac{1-3\xi_n^2+3i\xi_n-i\xi_n^3}{(\xi_n-i)^4(\xi_n+i)^3}fX_nd\xi_n\sigma(\xi')dx'\nonumber\\
&=-4if\Omega_3X_n\int_{\Gamma^+}\frac{1-3\xi_n^2+3i\xi_n-i\xi_n^3}{(\xi_n-i)^4(\xi_n+i)^3}d\xi_ndx'\nonumber\\
&=-4if\Omega_3X_n\frac{2\pi i}{3!}\bigg[\frac{1-3\xi_n^2+3i\xi_n-i\xi_n^3}{(\xi_n+i)^3}\bigg]^{(3)}\bigg|_{\xi_n=i}dx'\nonumber\\
&=fX_n\pi\Omega_3dx'.
\end{align*}
Finally, we get
\begin{align*}
\Phi_5=\bigg(-\frac{9}{8} h'(0)+fX_n\bigg)\pi\Omega_3dx'.
\end{align*}
Now $\Phi$ is the sum of the {\bf case a-I)}-{\bf case a-V)}. Therefore, we get
\begin{align}
\label{76}
\Phi=\sum_{i=1}^5\Phi_i=0.
\end{align}
By Theorem \ref{thm2} and (\ref{76}), we find that the boundary term vanishes. Thus, Theorem \ref{thmb1} holds.
\end{proof}
\section{The second type pseudo-differential perturbation of the Dirac operator $D+D^{-1}c(X)D$}
\label{section:3}
In this section, we define the second type pseudo-differential perturbation of the Dirac operator $D+D^{-1}c(X)D$, and prove the Kastler-Kalau-Walze type theorem of the operator $D+D^{-1}c(X)D$ on 4-dimensional manifolds with boundary.
\subsection{${\rm Wres}(D+D^{-1}c(X)D)^{-2}$ on manifolds without boundary}
Now, we define the second type pseudo-differential perturbation of the Dirac operator, set $B=D+D^{-1}c(X)D,$ then we get
\begin{align*}
B^2=[D+D^{-1}c(X)D]^2&=D^2+c(X)D+D^{-1}c(X)D^{2}-|X|^{2}+D^{-1}c[d(|X|^2)].
\end{align*}
We need to compute the total symbol $\sigma(x,\xi)$ of $B^{-2}$ up to order -4, with $B^2$ the
following sum of terms $B^2_k$ of order $k$:
$$
B^2=(B^2)_2+(B^2)_1+(B^2)_0+(B^2)_{<0}.
$$
\begin{align}\label{3322}
\sigma^{B^2}_2(x,\xi)&=|\xi|^2;\nonumber\\
\sigma^{B^2}_1(x,\xi)&=i(\Gamma^\mu-2\sigma^\mu)\xi_\mu+ic(\xi)c(X)+ic(X)c(\xi);\nonumber\\
\sigma^{B^2}_0(x,\xi)&=-(\partial^x\sigma_\mu+\sigma^\mu\sigma_\mu-\Gamma^\mu\sigma_\mu)+\frac{1}{4}s+ic(X)\gamma^\mu\sigma_\mu-|X|^2-|\xi|^{-2}c(\xi)c(X)(\Gamma^\mu-2\sigma^\mu)\xi_\mu\nonumber\\
&+i|\xi|^{-2}c(\xi)\gamma^\mu\sigma_\mu c(\xi)c(X)+|\xi|^{-4}[c(\xi)c(d{x_\mu})\partial_\mu^x(c(\xi))|\xi|^{2}-c(\xi)\partial_\mu^x(|\xi|^2)]c(X)-[\partial_\xi^\mu(|\xi|^{-2})c(\xi)\nonumber\\
&+|\xi|^{-2}\partial_\xi^\mu(c(\xi))][\partial_\mu^x(c(X))|\xi|^2+c(X)\partial_\mu^x(|\xi|^2)].
\end{align}
Next, we compute a parametrix $B^{-2}$ of $B^2$ up to order -4 using the above
recipe: this amounts to computing the parts $\sigma^{B^{-2}}_{-k},k=2,3,4,$ in the expansion of the
full symbol $\sigma$ of $B^2$ into terms of decreasing order:
\begin{align*}
\sigma^{B^{-2}}=\sigma^{B^{-2}}_{-2}+\sigma^{B^{-2}}_{-3}+\sigma^{B^{-2}}_{-4}+terms~of~order \leq -5.
\end{align*}
Application of (\ref{1111}) with $P=B^2$ and $Q=B^{-2}$ yields in the respective orders
$0,- 1,-2$ the recurrence relations:
\begin{align}\label{67890}
\sigma^{B^{-2}}_{-2}&=(\sigma^{B^{2}}_2)^{-1};\nonumber\\
\sigma^{B^{-2}}_{-3}&=-\sigma^{B^{-2}}_{-2}[\sigma^{B^{2}}_1\sigma^{B^{-2}}_{-2}-i\partial_{\xi}^\mu\sigma^{B^{2}}_2\partial_\mu^x\sigma^{B^{-2}}_{-2}];\nonumber\\
\sigma^{B^{-2}}_{-4}&=-\sigma^{B^{-2}}_{-2}[\sigma^{B^{2}}_1\sigma^{B^{-2}}_{-3}+\sigma^{B^{2}}_0\sigma^{B^{-2}}_{-2}-i\partial_{\xi}^\mu\sigma^{B^{2}}_1\partial_\mu^x\sigma^{B^{-2}}_{-2}-i\partial_{\xi}^\mu\sigma^{B^{2}}_2\partial_\mu^x\sigma^{B^{-2}}_{-3}].
\end{align}
Put (\ref{3322}) into (\ref{67890}), we have
\begin{align}\label{9998}
\sigma^{B^{-2}}_{-2}&=|\xi|^{-2};\nonumber\\
\sigma^{B^{-2}}_{-3}&=-|\xi|^{-2}[(ic(\xi)c(X)+ic(X)c(\xi)+i(\Gamma^\mu-2\sigma^\mu)\xi_\mu)|\xi|^{-2}-i\partial^\mu_\xi(|\xi|^2)\partial_\mu^x(|\xi|^{-2})];\nonumber\\
\sigma^{B^{-2}}_{-4}&=-|\xi|^{-6}\xi_\mu\xi_\nu(\Gamma^\mu-2\sigma^\mu)(\Gamma^\nu-2\sigma^\nu)-2|\xi|^{-8}\xi^\mu\xi_\alpha\xi_\beta(\Gamma^\nu-2\sigma^\nu)\partial^x_\mu g^{\alpha\beta}+|\xi|^{-4}(\partial^{x\mu}\sigma_\mu+\sigma^\mu\sigma_\mu\nonumber\\
&-\Gamma^\mu\sigma_\mu)-\frac{1}{4}|\xi|^{-4}s-2i|\xi|^{-2}\xi^\mu\partial^x_\mu\sigma^{B^{-2}}_{-3}+|\xi|^{-6}\xi_\alpha\xi_\beta(\Gamma^\mu-2\sigma^\mu)\partial^x_\mu g^{\alpha\beta}-|\xi|^{-6}\xi_\alpha\xi_\beta g^{\mu\nu}\partial^x_{\mu\nu} g^{\alpha\beta}\nonumber\\
&+2|\xi|^{-8}\xi_\alpha\xi_\beta\xi_\gamma\xi_\delta g^{\mu\nu}\partial^x_\mu g^{\alpha\beta}\partial^x_\nu g^{\gamma\delta}-|\xi|^{-6}(c(\xi)c(X)+c(X)c(\xi))(\Gamma^\mu-2\sigma^\mu)\xi_\mu+2|\xi|^{-8}(c(\xi)\nonumber\\
&c(X)+c(X)c(\xi))\xi^\mu\xi_\alpha\xi_\beta\partial^x_\mu g^{\alpha\beta}-|\xi|^{-6}[c(\xi)c(X)+c(X)c(\xi)]^2-|\xi|^{-6}(\Gamma^\mu-2\sigma^\mu)\xi_\mu(c(\xi)c(X)\nonumber\\
&+c(X)c(\xi))+i|\xi|^{-4}c(X)\gamma^\mu\sigma_\mu-|\xi|^{-4}|X|^2-|\xi|^{-6}c(\xi)c(X)(\Gamma^\mu-2\sigma^\mu)\xi_\mu+i|\xi|^{-6}c(\xi)\gamma^\mu\sigma_\mu \nonumber\\
&c(\xi)c(X)+|\xi|^{-6}c(\xi)c(d_{x_\mu})\partial_\mu^x(c(\xi))c(X)-|\xi|^{-8}c(\xi)\partial_\mu^x(|\xi|^2)c(X)-|\xi|^{-4}[\partial_\xi^\mu(|\xi|^{-2})c(\xi)\nonumber\\
&+|\xi|^{-2}\partial_\xi^\mu(c(\xi))][\partial_\mu^x(c(X))|\xi|^2+c(X)\partial_\mu^x(|\xi|^2)].
\end{align}
Regrouping the terms and inserting:
\begin{align*}
\partial^x_\mu\sigma^{B^{-2}}_{-3}&=2i|\xi|^{-6}\xi_\nu\xi_\alpha\xi_\beta(\Gamma^\nu-2\sigma^\nu)\partial^x_\mu g^{\alpha\beta}-i|\xi|^{-4}\xi_\nu\partial^x_\mu(\Gamma^\nu-2\sigma^\nu)+6i|\xi|^{-8}\xi_\nu\xi_\alpha\xi_\beta\xi_\gamma\xi_\delta\partial^x_\mu g^{\alpha\beta}\partial^x_\nu g^{\gamma\delta}\nonumber\\
&-2i|\xi|^{-6}\xi_\alpha\xi_\gamma\xi_\delta\partial^x_\mu g^{\nu\alpha}\partial^x_\nu g^{\gamma\delta}-2i|\xi|^{-6}\xi^\nu\xi_\gamma\xi_\delta\partial^x_{\mu\nu}g^{\gamma\delta}-i\partial^x_\mu [|\xi|^{-4}(c(\xi)c(X)+c(X)c(\xi))] .
\end{align*}
We get for $\sigma^{B^{-2}}_{-4}$ the sum of terms: $N_1-N_{10}$ and the followings
\begin{align*}
R_1&=-|\xi|^{-6}[c(\xi)c(X)+c(X)c(\xi)]\xi_\mu(\Gamma^\mu-2\sigma^\nu);\nonumber\\
R_2&=-|\xi|^{-6}\xi_\mu(\Gamma^\mu-2\sigma^\nu)[c(\xi)c(X)+c(X)c(\xi)];\nonumber\\
R_3&=-|\xi|^{-6}[c(\xi)c(X)+c(X)c(\xi)]^2;\nonumber\\
R_4&=2|\xi|^{-8}[c(\xi)c(X)+c(X)c(\xi)]\xi^\mu\xi_\alpha\xi_\beta\partial^x_\mu  g^{\alpha\beta};\nonumber\\
R_5&=-2|\xi|^{-2}\xi^\mu\partial^x_\nu [|\xi|^{-2}(c(\xi)c(X)+c(X)c(\xi))];\nonumber\\
R_6&=i|\xi|^{-4}c(X)\gamma^\mu\sigma_\mu-|\xi|^{-4}|X|^2-|\xi|^{-6}c(\xi)c(X)(\Gamma^\mu-2\sigma^\mu)\xi_\mu+i|\xi|^{-6}c(\xi)\gamma^\mu\sigma_\mu c(\xi)c(X);\nonumber\\
R_7&=|\xi|^{-6}c(\xi)c(d{x_\mu})\partial_\mu^x(c(\xi))c(X)-|\xi|^{-8}c(\xi)\partial_\mu^x(|\xi|^2)c(X);\nonumber\\
R_8&=-|\xi|^{-4}[\partial_\xi^\mu(|\xi|^{-2})c(\xi)+|\xi|^{-2}\partial_\xi^\mu(c(\xi))][\partial_\mu^x(c(X))|\xi|^2+c(X)\partial_\mu^x(|\xi|^2)].
\end{align*}
By (\ref{666}), we can get the noncommutative residue of the operator $D+D^{-1}c(X)D$ on 4-dimensional manifolds without boundary. Now we need to compute
\begin{align*}
{\rm Wres}(D+D^{-1}c(X)D)^{-2}=\int_{M}\int_{|\xi|=1}{\rm tr}(\sigma^{B^{-2}}_{-4})(x_0)\sigma(\xi)dx.
\end{align*}
Likewise, we only need to compute $\int_{|\xi|=1}{\rm tr}[\sum_{i=1}^{8}R_i](x_0)\sigma(\xi).$

$\mathbf{(1):}$ By (\ref{a26})-(\ref{lll}), we obtain that the result of this term $R_1$, $R_2$ and $R_4$ disappear in normal coordinates.
$\int_{|\xi|=1}{\rm tr}(R_3)(x_0)\sigma(\xi)=-2\pi^2|X|^2{\rm tr}[{\rm \texttt{id}}]$, $\int_{|\xi|=1}{\rm tr}(R_5)(x_0)\sigma(\xi)=2\pi^2div_M(X){\rm tr}[{\rm \texttt{id}}]$.

$\mathbf{(2):}$
\begin{align*}
&{\rm tr}(R_6)(x_0)_{|\xi|=1}=-|X|^2{\rm tr}[{\rm \texttt{id}}],
\end{align*}
then
\begin{align*}
\int_{|\xi|=1}{\rm tr}(R_6)(x_0)\sigma(\xi)&=-2\pi^2|X|^2{\rm tr}[{\rm \texttt{id}}].
\end{align*}

$\mathbf{(3):}$
\begin{align*}
&{\rm tr}(R_7)_{|\xi|=1}=c(\xi)c(d{x_\mu})\partial_\mu^x(c(\xi))c(X)-c(\xi)\partial_\mu^x(|\xi|^2)c(X).
\end{align*}
Using the facts:
$\partial_\mu^x(c(\xi))(x_0)=0$ and $\partial_\mu^x(|\xi|^2)(x_0)=0,$ then in normal coordinates, the result of this term $R_7$ disappears.

$\mathbf{(4):}$
By $\partial^\mu_\xi(|\xi|^{-2})=-2|\xi|^{-4}\xi^\mu,$ $\partial^\mu_\xi(c(\xi))=c(dx_\mu)$ and $\partial^\mu_x(c(X))(x_0)=\sum_{j=1}^n\partial^\mu_x(X_j)c(\partial_{x_j})(x_0)=\sum_{j=1}^n\partial^\mu_x(X_j)c(e_j)(x_0)$, we have
\begin{align*}
{\rm tr}(R_8)(x_0)&=-[-2\xi^\mu c(\xi)+c(dx_\mu)][\sum_{j=1}^n\partial^\mu_x(X_j)c(e_j)](x_0)\nonumber\\
&=\sum_\mu[\sum_l2\xi_\mu\xi_lc(dx_l)-c(dx_\mu)][\sum_{j=1}^n\partial^\mu_x(X_j)c(dx_j)](x_0)\nonumber\\
&=[-\sum_{\mu l}2\xi_\mu\xi_l\partial^\mu_x(X_j)g^{lj}+\sum_{\mu j}\partial^\mu_x(X_j)g^{\mu l}(x_0)]{\rm tr}[{\rm \texttt{id}}],
\end{align*}
further, by $\int_{S^3}\xi_\mu\xi_l\sigma(\xi)=\frac{1}{4}\delta_{\mu l}Vol_{S^3}=\frac{1}{2}\pi^2\delta_{\mu l},$ we obtain
\begin{align*}
\int_{|\xi|=1}{\rm tr}(R_8)(x_0)\sigma(\xi)&=-\frac{1}{2}\int_{|\xi|=1}\sum_j\partial_{x_j}(X_j){\rm tr}[{\rm \texttt{id}}]\sigma(\xi)+\int_{|\xi|=1}\sum_j\partial_{x_j}(X_j){\rm tr}[{\rm \texttt{id}}]\sigma(\xi)\nonumber\\
&=\frac{1}{2}\sum_j\partial_{x_j}(X_j){\rm tr}[{\rm \texttt{id}}]Vol_{S^3}\nonumber\\
&=\pi^2div_M(X){\rm tr}[{\rm \texttt{id}}].
\end{align*}
Thus
\begin{align*}
\int_{|\xi|=1}{\rm tr}(\sigma^{B^{-2}}_{-4})\sigma(\xi)=\bigg(-4\pi^2|X|^2+3\pi^2div_M(X)+\frac{1}{12}s\bigg){\rm tr}[{\rm \texttt{id}}].
\end{align*}
Then, by (\ref{666}), we have the following result.
\begin{thm}\label{thm82} Let $M$ be a $4$-dimensional oriented
compact spin manifold without boundary, then we get the noncommutative residue of the operator $D+D^{-1}c(X)D$
\begin{align*}
&{\rm Wres}(D+D^{-1}c(X)D)^{-2}=4\int_{M}\bigg(-4\pi^2|X|^2+3\pi^2div_M(X)+\frac{1}{12}s\bigg) d{\rm Vol_{M}}.
\end{align*}
\end{thm}

\subsection{${\rm \widetilde{Wres}}[\pi^+(D+D^{-1}c(X)D)^{-1}\circ \pi^+(D+D^{-1}c(X)D)^{-1}]$  on manifolds with boundary}
Some symbols for the second type pseudo-differential perturbation of the Dirac operator are given by (\ref{3322})
and (\ref{9998}). Then we get the following lemmas
\begin{lem}\label{2lem2}Some symbols of positive order for the second type pseudo-differential perturbation of the Dirac
operator are as follows.
\begin{align*}
\sigma_1(D+D^{-1}c(X)D)&=\sigma_1(D)=ic(\xi); \nonumber\\
\sigma_0(D+D^{-1}c(X)D)&=\sigma_0(D)-c(X)+2\frac{c(\xi)\xi(X)}{|\xi|^2}=-\frac{1}{4}\sum_{i,s,t}\omega_{s,t}(e_i)c(e_i)c(e_s)c(e_t)-c(X)+2\frac{c(\xi)\xi(X)}{|\xi|^2}.
\end{align*}
\end{lem}
Then by the composition formula of pseudo-differential operators, we have
\begin{lem}\label{2lem3}Some symbols of negative order for the second type pseudo-differential perturbation of the Dirac
operator are as follows.
\begin{align*}
\sigma_{-1}({D+D^{-1}c(X)D})^{-1}&=\frac{ic(\xi)}{|\xi|^2};\nonumber\\
\sigma_{-2}({D+D^{-1}c(X)D})^{-1}&=\frac{c(\xi)\sigma_{0}({D+D^{-1}c(X)D})c(\xi)}{|\xi|^4}+\frac{c(\xi)}{|\xi|^6}\sum_jc(dx_j)
\Big[\partial_{x_j}(c(\xi))|\xi|^2-c(\xi)\partial_{x_j}(|\xi|^2)\Big].
\end{align*}
\end{lem}
Now, by (\ref{b7}) and (\ref{b8}), we obtain lower dimensional volumes of spin manifolds with boundary for the operator $D+D^{-1}c(X)D$
\begin{align}
\label{b1P4}
&\widetilde{{\rm Wres}}[\pi^+(D+D^{-1}c(X)D)^{-1}\circ \pi^+(D+D^{-1}c(X)D)^{-1}]\nonumber\\
&=\int_M\int_{|\xi|=1}{\rm
tr}_{S(TM)\bigotimes\mathbb{C}}[\sigma_{-4}(D+D^{-1}c(X)D)^{-2}]\sigma(\xi)dx+\int_{\partial M}\widetilde{\Phi},
\end{align}
where
\begin{align}
\label{b1P5}
\widetilde{\Phi} &=\int_{|\xi'|=1}\int^{+\infty}_{-\infty}\sum^{\infty}_{j, k=0}\sum\frac{(-i)^{|\alpha|+j+k+1}}{\alpha!(j+k+1)!}
\times{\rm
tr}_{S(TM)\bigotimes\mathbb{C}}[\partial^j_{x_n}\partial^\alpha_{\xi'}\partial^k_{\xi_n}\sigma^+_{r}(D+D^{-1}c(X)D)^{-1}(x',0,\xi',\xi_n)
\nonumber\\
&\times\partial^\alpha_{x'}\partial^{j+1}_{\xi_n}\partial^k_{x_n}\sigma_{l}(D+D^{-1}c(X)D)^{-1}(x',0,\xi',\xi_n)]d\xi_n\sigma(\xi')dx',
\end{align}
and the sum is taken over $r+l-k-j-|\alpha|=-3,~~r\leq -1,~~l\leq-1$.

Similar to the case of the first type pseudo-differential perturbation of the Dirac operator,  we have obtained the interior term of $\widetilde{{\rm Wres}}[\pi^+(D+D^{-1}c(X)D)^{-1}\circ\pi^+(D+D^{-1}c(X)D)^{-1}]$ in Theorem \ref{thm82}. After computing $\int_{\partial M} \widetilde{\Phi}$, we get the following theorem.
\begin{thm}\label{thmbP1}
Let $M$ be a $4$-dimensional oriented
compact spin manifold without boundary, then we get the Kastler-Kalau-Walze type theorem of the operator $D+D^{-1}c(X)D$
\begin{align*}
&\widetilde{{\rm Wres}}[\pi^+(D+D^{-1}c(X)D)^{-1}\circ\pi^+(D+D^{-1}c(X)D)^{-1}]\nonumber\\
&=4\int_{M}\bigg(-4\pi^2|X|^2+3\pi^2div_M(X)+\frac{1}{12}s\bigg) d{\rm Vol_{M}}.
\end{align*}
In particular, the boundary term vanishes.
\end{thm}
\begin{proof}
\indent When $n=4$, the sum is taken over $
r+l-k-j-|\alpha|=-3,~~r\leq -1,~~l\leq-1,$ then we have the following five cases:
~\\
\noindent  {\bf case b-I)}~$r=-1,~l=-1,~k=j=0,~|\alpha|=1$.\\
\noindent By (\ref{b1P5}), we get
\begin{align*}
\widetilde{\Phi}_1&=-\int_{|\xi'|=1}\int^{+\infty}_{-\infty}\sum_{|\alpha|=1}
 {\rm tr}[\partial^\alpha_{\xi'}\pi^+_{\xi_n}\sigma_{-1}(D+D^{-1}c(X)D)^{-1}\\
 &\times
 \partial^\alpha_{x'}\partial_{\xi_n}\sigma_{-1}(D+D^{-1}c(X)D)^{-1}](x_0)d\xi_n\sigma(\xi')dx'.
\end{align*}
 \noindent  {\bf case b-II)}~$r=-1,~l=-1,~k=|\alpha|=0,~j=1$.\\
\noindent By (\ref{b1P5}), we get
\begin{align*}
\widetilde{\Phi}_2&=-\frac{1}{2}\int_{|\xi'|=1}\int^{+\infty}_{-\infty} {\rm
tr} [\partial_{x_n}\pi^+_{\xi_n}\sigma_{-1}(D+D^{-1}c(X)D)^{-1}\\
&\times
\partial_{\xi_n}^2\sigma_{-1}(D+D^{-1}c(X)D)^{-1}](x_0)d\xi_n\sigma(\xi')dx'.
\end{align*}
\noindent  {\bf case b-III)}~$r=-1,~l=-1,~j=|\alpha|=0,~k=1$.\\
\noindent By (\ref{b1P5}), we get
\begin{align*}
\widetilde{\Phi}_3&=-\frac{1}{2}\int_{|\xi'|=1}\int^{+\infty}_{-\infty}
{\rm tr} [\partial_{\xi_n}\pi^+_{\xi_n}\sigma_{-1}({D+D^{-1}c(X)D})^{-1}\\
&\times
\partial_{\xi_n}\partial_{x_n}\sigma_{-1}({D+D^{-1}c(X)D})^{-1}](x_0)d\xi_n\sigma(\xi')dx'.
\end{align*}
By Lemma \ref{lem3} and Lemma \ref{2lem3}, we have $\sigma_{-1}({D+D^{-1}c(X)D})^{-1}=\sigma_{-1}({D+c(X)D^{-1}fD})^{-1}$.
Then
 $$\widetilde{\Phi}_i=\Phi_i,~~~i=1,2,3.$$
\noindent  {\bf case b-IV)}~$r=-2,~l=-1,~k=j=|\alpha|=0$.\\
\noindent By (\ref{b1P5}), we get
\begin{align*}
\widetilde{\Phi}_4&=-i\int_{|\xi'|=1}\int^{+\infty}_{-\infty}{\rm tr} [\pi^+_{\xi_n}\sigma_{-2}({D+D^{-1}c(X)D})^{-1}\\
&\times
\partial_{\xi_n}\sigma_{-1}({D+D^{-1}c(X)D})^{-1}](x_0)d\xi_n\sigma(\xi')dx'.
\end{align*}
 By Lemma \ref{lem3}, we have
\begin{align*}
&\sigma_{-2}({D+D^{-1}c(X)D})^{-1}(x_0)\nonumber\\
&=\frac{c(\xi)\sigma_{0}(D+D^{-1}c(X)D)(x_0)c(\xi)}{|\xi|^4}+\frac{c(\xi)}{|\xi|^6}c(dx_n)
[\partial_{x_n}[c(\xi')](x_0)|\xi|^2-c(\xi)h'(0)|\xi|^2_{\partial
M}].
\end{align*}
Moreover, by (\ref{b3}) and (\ref{b4}), we obtain
\begin{align*}
&\pi^+_{\xi_n}\sigma_{-2}(D+D^{-1}c(X)D)^{-1}|_{|\xi'|=1}\nonumber\\
&=\pi^+_{\xi_n}\Big[\frac{c(\xi)Q(x_0)c(\xi)}{(1+\xi_n^2)^2}\Big]-\pi^+_{\xi_n}
\Big[\frac{c(\xi)c(X)c(\xi)}{(1+\xi_n^2)^2}\Big]
+\pi^+_{\xi_n}\Big[\frac{c(\xi)c(dx_n)\partial_{x_n}[c(\xi')](x_0)}{(1+\xi_n^2)^2}-h'(0)\frac{c(\xi)c(dx_n)c(\xi)}{(1+\xi_n^{2})^3}\Big]\nonumber\\
&+\pi^+_{\xi_n}\Big[-2\frac{\xi(X)c(\xi)}{(1+\xi_n^2)^2}\Big]\nonumber\\
&:=C_1-C_2-\frac{1}{f}C_3+C_4,
\end{align*}
where
\begin{align}\label{5p3}
C_4&=\frac{2+i\xi_n}{2(\xi_n-i)^2}\sum_{j=1}^{n-1}\xi_jX_jc(\xi')+\frac{i}{2(\xi_n-i)^2}X_nc(\xi')+\frac{i}{2(\xi_n-i)^2}\sum_{j=1}^{n-1}\xi_jX_jc(dx_n)+\frac{i\xi_n}{2(\xi_n-i)^2}X_nc(dx_n).
\end{align}
Since
\begin{align}\label{5p0}
\partial_{\xi_n}\sigma_{-1}(D+D^{-1}c(X)D)^{-1}=i\left[\frac{c(dx_n)}{1+\xi_n^2}-\frac{2\xi_nc(\xi')+2\xi_n^2c(dx_n)}{(1+\xi_n^2)^2}\right].
\end{align}
Then by (\ref{5p3}) and (\ref{5p0}), we have
\begin{align*}{\rm tr }[C_4\times\partial_{\xi_n}\sigma_{-1}(D+D^{-1}c(X)D)^{-1}]|_{|\xi'|=1}=
\frac{-4\xi_n^2+8i\xi_n+2}{(\xi_n-i)^4(\xi_n+i)^2}\sum_{j=1}^{n-1}\xi_jX_j+\frac{-2\xi_n^3-2\xi_n}{(\xi_n-i)^4(\xi_n+i)^2}X_n.
\end{align*}
When $i<n,~\int_{|\xi'|=1}\xi_{i_{1}}\xi_{i_{2}}\cdots\xi_{i_{2d+1}}\sigma(\xi')=0$,
so $\sum_{j=1}^{n-1}\xi_jX_j$ has no contribution for computing {\rm case~b-IV)}, we obtain
\begin{align*}
&-i\int_{|\xi'|=1}\int^{+\infty}_{-\infty}{\rm tr} [C_4\times
\partial_{\xi_n}\sigma_{-1}(D+D^{-1}c(X)D)^{-1}](x_0)d\xi_n\sigma(\xi')dx'\nonumber\\
&=-i\int_{|\xi'|=1}\int^{+\infty}_{-\infty}\frac{-2\xi_n^3-2\xi_n}{(\xi_n-i)^4(\xi_n+i)^2}X_nd\xi_n\sigma(\xi')dx'\nonumber\\
&=2iX_n\Omega_3\int_{\Gamma^+}\frac{\xi_n^3+\xi_n}{(\xi_n-i)^4(\xi_n+i)^2}d\xi_ndx'\nonumber\\
&=2iX_n\Omega_3\frac{2\pi i}{3!}\bigg[\frac{\xi_n^3+\xi_n}{(\xi_n+i)^2}\bigg]^{(3)}\bigg|_{\xi_n=i}dx'\nonumber\\
&=-\frac{1}{2}X_n\Omega_3dx'.
\end{align*}
Therefore, we get
\begin{align*}
\Phi_4=\bigg(\frac{9}{8} h'(0)+\frac{1}{2}X_n\bigg)\pi\Omega_3dx'.
\end{align*}
\noindent {\bf  case b-V)}~$r=-1,~l=-2,~k=j=|\alpha|=0$.\\
By (\ref{b1P5}), we get
\begin{align*}
\widetilde{\Phi}_5=-i\int_{|\xi'|=1}\int^{+\infty}_{-\infty}{\rm tr} [\pi^+_{\xi_n}\sigma_{-1}(D+D^{-1}c(X)D)^{-1}\times
\partial_{\xi_n}\sigma_{-2}(D+D^{-1}c(X)D)^{-1}](x_0)d\xi_n\sigma(\xi')dx'.
\end{align*}
By (\ref{b3}), (\ref{b4}) and Lemma \ref{lem3}, we have
\begin{align}\label{6p2}
\pi^+_{\xi_n}\sigma_{-1}(D+D^{-1}c(X)D)^{-1}|_{|\xi'|=1}=\frac{c(\xi')+ic(dx_n)}{2(\xi_n-i)}.
\end{align}
Since
\begin{align*}
&\sigma_{-2}(D+D^{-1}c(X)D)^{-1}(x_0)\nonumber\\
&=\frac{c(\xi)\sigma_{0}(D+D^{-1}c(X)D)(x_0)c(\xi)}{|\xi|^4}+\frac{c(\xi)}{|\xi|^6}c(dx_n)
\bigg[\partial_{x_n}[c(\xi')](x_0)|\xi|^2-c(\xi)h'(0)|\xi|^2_{\partial_
M}\bigg].
\end{align*}
Moreover
\begin{align*}
&\partial_{\xi_n}\sigma_{-2}(D+D^{-1}c(X)D)^{-1}(x_0)|_{|\xi'|=1}\nonumber\\
&=
\partial_{\xi_n}\bigg\{\frac{c(\xi)[Q(x_0)
+c(X)]c(\xi)}{|\xi|^4}+\frac{c(\xi)}{|\xi|^6}c(dx_n)[\partial_{x_n}[c(\xi')](x_0)|\xi|^2-c(\xi)h'(0)]\bigg\}\nonumber\\
&=\partial_{\xi_n}\bigg\{\frac{c(\xi)Q(x_0)c(\xi)}{|\xi|^4}+\frac{c(\xi)}{|\xi|^6}c(dx_n)[\partial_{x_n}[c(\xi')](x_0)|\xi|^2-c(\xi)h'(0)]\bigg\}-\partial_{\xi_n}\bigg(\frac{c(\xi)c(X)c(\xi)}{|\xi|^4}\bigg)\nonumber\\
&+\partial_{\xi_n}\bigg(-2\frac{\xi(X)c(\xi)}{|\xi|^4}\bigg).
\end{align*}
By computations, we have
\begin{align}\label{6p7}
&\partial_{\xi_n}\bigg(-2\frac{\xi(X)c(\xi)}{|\xi|^4}\bigg)\nonumber\\
&=-\frac{8\xi_n}{(1+\xi_n^2)^3}\sum_{j=1}^{n-1}\xi_jX_jc(\xi')+\frac{6\xi_n^2-2}{(1+\xi_n^2)^3}X_nc(\xi')+\frac{6\xi_n^2-2}{(1+\xi_n^2)^3}\sum_{j=1}^{n-1}\xi_jX_jc(dx_n)+\frac{4\xi_n^3-4\xi_n}{(1+\xi_n^2)^3}X_nc(dx_n).
\end{align}
By (\ref{6p2}) and (\ref{6p7}), we have
\begin{align}\label{7P3}
&{\rm tr}\bigg[\pi^+_{\xi_n}\sigma_{-1}(D+D^{-1}c(X)D)^{-1}\times
\partial_{\xi_n}\bigg(-2\frac{\xi(X)c(\xi)}{|\xi|^4}\bigg)\bigg](x_0)\nonumber\\
&=-4\frac{3i\xi_n^2+4\xi_n-i}{(\xi_n-i)^4(\xi_n+i)^3}\sum_{j=1}^{n-1}\xi_jX_j-4\frac{2i\xi_n^3+3\xi_n^2-2i\xi_n-1}{(\xi_n-i)^4(\xi_n+i)^3}X_n.
\end{align}
When $i<n,~\int_{|\xi'|=1}\xi_{i_{1}}\xi_{i_{2}}\cdots\xi_{i_{2d+1}}\sigma(\xi')=0$ and $\sum_{j=1}^{n-1}\xi_jX_j$ has no contribution for computing {\rm case~b-V)}, we have
\begin{align*}
&-i\int_{|\xi'|=1}\int^{+\infty}_{-\infty}{\rm tr}[\pi^+_{\xi_n}\sigma_{-1}(D+D^{-1}c(X)D)^{-1}\times
\partial_{\xi_n}\bigg(-2\frac{\xi(X)c(\xi)}
{|\xi|^4}\bigg)](x_0)d\xi_n\sigma(\xi')dx'\nonumber\\
&=-i\int_{|\xi'|=1}\int^{+\infty}_{-\infty}-4\frac{2i\xi_n^3+3\xi_n^2-2i\xi_n-1}{(\xi_n-i)^4(\xi_n+i)^3}X_nd\xi_n\sigma(\xi')dx'\nonumber\\
&=4iX_n\Omega_3\int_{\Gamma^+}\frac{2i\xi_n^3+3\xi_n^2-2i\xi_n-1}{(\xi_n-i)^4(\xi_n+i)^3}d\xi_ndx'\nonumber\\
&=4iX_n\Omega_3\frac{2\pi i}{3!}\bigg[\frac{2i\xi_n^3+3\xi_n^2-2i\xi_n-1}{(\xi_n+i)^3}\bigg]^{(3)}\bigg|_{\xi_n=i}dx'\nonumber\\
&=\frac{1}{2}X_n\Omega_3dx'.
\end{align*}
Then
\begin{align*}
\widetilde{\Phi}_5=\bigg(-\frac{9}{8} h'(0)-\frac{1}{2}X_n\bigg)\pi\Omega_3dx'.
\end{align*}
Now $\widetilde{\Phi}$ is the sum of the {\bf  (case b-I)}-{\bf  (case b-V)}. Therefore, we get
\begin{align}\label{766}
\widetilde{\Phi}=\sum_{i=1}^5\widetilde{\Phi}_i=0.
\end{align}
By Theorem \ref{thm82} and (\ref{766}), we find that the boundary term vanishes. Thus, Theorem \ref{thmbP1} holds.
\end{proof}
\section{Examples}
\label{section:4}
Let $(\mathcal{A},\mathcal{H},D)$ be an $n$-summable unital spectral triple, where $\mathcal{A}$ is a noncommutative algebra with involution, acting in the Hilbert space $\mathcal{H}$ while $D$ is a Dirac operator, which is self-adjoint
operator with compact resolvent and such that
$[D,a]$ is bounded $\forall a \in \mathcal{A}$. We assume that there exists a generalised algebra of
 pseudo-differential operators, which contains $\mathcal{A},D,$ $D^l$ for $l\in \mathbb{Z}$ and there exists a tracial state $\mathcal{W}$  on it, called a noncommutative residue. Moreover, we assume that the
 noncommutative residue identically vanishes on $TD^{-k}$ for any $k>2m$ and a zero-order
 operator $T$. Let $a,b\in \mathcal{A},$ we define $\widetilde{D}=D+[D,a]D^{-1}bD$ and $\widehat{D}=D+D^{-1}[D,a]D,$ we call that $\widetilde{D}$ and $\widehat{D}$ are the pseudo-differential perturbations of $D,$ which generalize the pseudo-differential perturbations of the Dirac operator in Theorem \ref{thm1111} and Theorem \ref{thm2222}.
\begin{exam}{\bf Hodge-Dirac triple}\\
\indent Start with a purely classical example of a Hodge-Dirac spectral triple: $(C^\infty(M),d+\delta,\Gamma(\wedge^*T^*M))$, where $M$ is an oriented closed Riemannian manifold and a grading $\Gamma$, the de Rham derivative $d$ is an elliptic differential operator on $C^\infty(M;\wedge^*T^*M)$ and where $\delta$ is
 an operator adjoint to $d$.
\end{exam}
\begin{exam}{\bf Almost commutative $M\times \mathbb{Z}_2$}\\
\indent We assume that $M$ is a closed spin manifold and $(C^\infty(M),D,\mathcal{H})$ is an even spectral triple of dimension $n=2m$ with the standard Dirac
 operator $D$ and a grading $\Gamma$. We consider the usual double-sheet spectral triple, $(C^\infty(M)\otimes \mathbb{C}^2,\widetilde{D},L^2(M,S(TM)))\oplus L^2(M,S(TM)))$ with the Dirac operator,
 \[
\widetilde{D}=\begin{bmatrix}
 D&0\\
0&0 \\
\end{bmatrix}+\Gamma\otimes \begin{bmatrix}
0&\Phi\\
\Phi^*&0 \\
\end{bmatrix}=
\begin{bmatrix}
 D&\Gamma\Phi\\
\Gamma\Phi^*&D \\
\end{bmatrix},
\]
 where $\Phi\in \mathbb{C}.$ Then\[
\begin{bmatrix}
    \widetilde{D},&
    \begin{bmatrix}  f_1&0\\
0& f_2 \end{bmatrix}
\end{bmatrix}
=
\begin{bmatrix}
c(df_1)&\Phi(f_2-f_1)\Gamma\\
    \Phi^*(f_1-f_2)\Gamma& c(df_2)
\end{bmatrix},\]
\end{exam}
where $f_1,f_2\in C^\infty(M).$
\begin{exam}{\bf Conformally rescaled noncommutative tori}\\
\indent Consider the spectral triple on the noncommutative $n$-torus, $(\mathcal{A},\mathcal{H},D_k)$, where $\mathcal{A}=C^\infty(\mathbb{T}^n_\theta)$ is the $\delta_j$-smooth subalgebra for
 the standard derivations $\delta_j$, $j=1,\cdot\cdot\cdot,n$. Next, $\mathcal{H}=V(m)\otimes L^2(\mathbb{T}^n_\theta,\tau)$, where $V(m)$ is the spinor space and $\tau$ is the standard
 trace which annihilates $\delta_j$. Also, $D_k=kDk$ is the conformal rescaling of the standard (flat) Dirac operator $D=\sum_j\gamma^j\delta_j$ with $\gamma^j$ being the usual gamma matrices and the conformal factor $k>0$ is from $\mathcal{A}$. We use the calculus of the pseudo-differential operators over  noncommutative tori. Then we have the analogue of the Wodzicki residue.
\end{exam}
In the three examples presented above, we can consider the corresponding pseudo-differential operators $\widetilde{D},\widehat{D}$. We will compute the similar residue of $\widetilde{D}^{-2}$ and $\widehat{D}^{-2}$ in above 4-dimensional examples and get simmilar Theorem \ref{thm1111} and Theorem \ref{thm2222} in the above examples in the future.
\section*{Declarations}
\textbf{Ethics approval and consent to participate:} Not applicable.

\textbf{Consent for publication:} Not applicable.

\textbf{Availability of data and materials:} The authors confrm that the data supporting the fndings of this study are available within the article.

\textbf{Competing interests:} The authors declare no competing interests.

\textbf{Funding:} This research was funded by National Natural Science Foundation of China: No.11771070.

\textbf{Author Contributions:} All authors contributed to the study conception and design. Material preparation,
data collection and analysis were performed by TW and YW. The frst draft of the manuscript was written
by TW and all authors commented on previous versions of the manuscript. All authors read and approved
the final manuscript.

\section*{Acknowledgements}
This work was supported by NSFC. 11771070 and the Fundamental Research Funds for the Central Universities N2405015.
 The authors thank the referee for his (or her) careful reading and helpful comments.

\end{document}